\documentclass{amsart}

\usepackage[utf8]{inputenc} 		
\usepackage[T1]{fontenc}	 		
\usepackage[british]{babel}			

\usepackage{amsmath,amsthm,amssymb}
\usepackage{mathtools,mathrsfs}

\usepackage{tikz-cd}				
\usetikzlibrary{positioning}		
\usetikzlibrary{calc}				

\usepackage{lscape}					
\usepackage{multirow}
\usepackage{array}					

\newcolumntype{C}{>{$}c<{$}}
\newcolumntype{L}{>{$}l<{$}}
\newcolumntype{R}{>{$}r<{$}}

\usepackage{verbatim}				
\usepackage{subcaption}				
\usepackage{hyperref}				

\usepackage{packages/algebra}		
\usepackage{packages/nicetikz}		
\usepackage{packages/singtheo}		
\usepackage{packages/fonts}

\newcommand\T{\rule{0pt}{2.6ex}}		
\newcommand\B{\rule[-1.2ex]{0pt}{0pt}} 	
\newcommand\M{\T\B} 					
\usepackage{makecell}					

\newtheorem{conjecture}{Conjecture}[section]
\newtheorem{corollary}{Corollary}[section]
\newtheorem{definition}{Definition}[section]
\newtheorem{lemma}[corollary]{Lemma}
\newtheorem{proposition}[corollary]{Proposition}
\newtheorem{theorem}[corollary]{Theorem}

\newtheorem{example}{Example}[section]
\newtheorem{remark}{Remark}[section]

\title{Singularities of Frontal Surfaces}
\author{C. Mu\~noz-Cabello, J.J. Nu\~no-Ballesteros, R. Oset Sinha}

\address{Departament de Matem\`{a}tiques,
Universitat de Val\`encia, Campus de Burjassot, 46100 Burjassot,
Spain}
\email{chmuca@alumni.uv.es}
\email{Raul.Oset@uv.es}

\address{Departament de Matemàtiques, Universitat de Val\`encia, Campus de Burjassot, 46100 Burjassot SPAIN.
Departamento de Matemática, Universidade Federal da Paraíba CEP 58051-900, João Pessoa - PB, BRAZIL}
\email{Juan.Nuno@uv.es}

\thanks{Work of Juan J. Nuño-Ballesteros and R. Oset Sinha partially supported by Grant PGC2018-094889-B-100 funded by MCIN/AEI/ 10.13039/501100011033 and by ``ERDF A way of making Europe"}


\begin{document}        
\maketitle

\begin{abstract}
We consider singularities of frontal surfaces of corank one and finite frontal codimension.
We look at the classification under $\mathscr A$-equivalence and introduce the notion of frontalisation for singularities of fold type.
We define the cuspidal and the transverse double point curves and prove that the frontal has finite codimension if and only if both curves are reduced.
Finally, we also discuss about the frontal versions of the Marar-Mond formulas and the Mond's conjecture.
\end{abstract}

\section{Introduction}
	In this paper we are interested in the local behaviour of complex analytic mappings $f\colon N\to Z$ of frontal type, where $N$ and $Z$ are complex analytic manifolds of dimensions 2 and 3 respectively. In general, a mapping $f\colon N\to Z$ is called a frontal if it admits a Legendrian lifting $\tilde f:N\to PT^*Z$, where $PT^*Z$ is the projectivised cotangent bundle of $Z$. Roughly speaking, this means that the image $f(N)$ has a well defined tangent hyperplane at each point $f(x)$, with $x\in N$. 

Singularities of frontals were considered for the first time by Zakalyukin and Kurbatskii in \cite{ZakalyukinKurbatskii} and they 
are a natural generalisation of wave fronts, which occur in the particular case that the Legendrian lifting $\tilde f$ is an immersion. There has been a great interest for frontals in the last decades, specially in the $C^\infty$ real category and looking at differential geometric properties. The fact that you have a well defined tangent plane everywhere provides a nice starting point if you want to extend things like first or second fundamental forms, curvatures, etc to submanifolds with singularities (see for instance \cite{ChenPeiTakahashi, MartinsNuno, MurataUmehara, OsetSaji, SUYGaussBonnet}).

In a forthcoming paper \cite{Frontals}, we will develop the general Thom-Mather theory of frontals of any dimension, but restricted to corank one singularities. Our approach is based on Ishikawa's work \cite{Ishikawa_Legendrian} about stability and infinitesimal deformations of integral mappings $\tilde f\colon N\to PT^*Z$ under Legendrian equivalence, 
although we want to understand the singularity downstairs, at the level of the frontal, rather than upstairs, at the level of the Legendrian lifting. In particular, we consider infinitesimal deformations which come from frontal unfoldings of $f$ itself. The corank one assumption is a technical but necessary condition in order to be able to apply Ishikawa's theory.

In our case, by taking local charts in $N$ and $Z$ we can restate our problem in terms of classification of frontal map germs $f\colon(\mb{C}^2,S)\to(\mb{C}^3,0)$ of corank one under $\mathscr A$-equivalence, that is, under holomorphic coordinate changes in the source and the target. It is well known that the only frontal stable  singularities of surfaces are cuspidal edges, swallowtails, folded Whitney umbrellas or their transverse self-intersections (see Figure \ref{fig: stable_frontals}). By the frontal version of the Mather-Gaffney criterion (see Theorem \ref{thm: frontal mather gaffney}), if $f$ has finite frontal codimension, then it has isolated instability and hence, the only singularities outside the origin are cuspidal edges or transverse double points. Our main goal is to study the geometry and the invariants of this type of frontal surface singularities.

In Section 3, we look at the classification of simple frontal singularities. We introduce the notion of frontalisation of a germ of fold type and we deduce that the classification of frontals of fold type is closely related to the classification of $\mathscr A$-simple singularities obtained by Mond in \cite{Mond_Classification}. In particular, all simple germs of fold type in Mond's classification $S_k,B_k,C_k$ and $F_4$ have a frontal version with the same frontal codimension.

In Section 4, we define the cuspidal curve $C(f)$ and the transverse double point curve $D_+(f)$. These two curves are considered with a certain analytic structure and they have the property that $f$ has finite frontal codimension if and only if both curves are reduced. We recall that $f$ has finite $\mathscr A_e$-codimension if and only if the double point curve $D(f)$ is reduced (see \cite{MararNunoPenafort}). If $f$ has corank one, then we can assume that is given by $f(x,y)=(x,p(x,y),q(x,y))$, for some function germs $p$ and $q$. The space $D^2(f)$ is defined in $\mb{C}^3$ by the divided differences:
\[
\frac{p(x,y')-p(x,y)}{y'-y}=\frac{q(x,y')-q(x,y)}{y'-y}=0,
\]
and the projection to $\mb{C}^2$ given by $(x,y,y')\mapsto(x,y)$ is precisely the double point curve $D(f)$. As set germs, $D(f)=C(f)\cup D_+(f)$. However when $f$ is a frontal, $D(f)$ has a non reduced equation $p_y^2\tau=0$, where $p_y=0$ is the equation of $D(f)$ and $\tau=0$ is the equation of $D_+(f)$.

Finally, in Section 5 we consider a frontal stable perturbation $f_s$ of $f$. This always exists and is well defined when $f$ has finite frontal codimension.  We get invariants which count the number of 0-stable singularities of $f_s$ which we call $S,W,K$ and $T$ and correspond to the number of swallowtails, folded Whitney umbrellas, cuspidal double points and triple points, respectively. We give algebraic formulas to compute these invariants in terms of some algebras related to $f$ and prove the frontal version of the Marar-Mond formulas (see \cite{MararMond_Formula}). Moreover, as it happens with germs of finite $\mathscr A_e$-codimension, the image of $f_s$ has the homotopy type of a wedge of $2$-spheres and the number of such spheres is called the frontal Milnor number, denoted by $\mu_\mathscr F(f)$. This is  analogous of the image Milnor number $\mu_I(f)$ defined by Mond  in \cite{Mond_VanishingCycles} in the context of finite $\mathscr A_e$-codimension. We discuss some basic properties of this frontal Milnor number, and propose a frontal version of Mond's conjecture \cite{Mond_Conjecture}, which states that the image Milnor number of $f\colon (\mb{C}^n,0) \to (\mb{C}^{n+1},0)$ is greater than or equal to its $\ms{A}_e$-codimension for all values of $n$ such that $(n,n+1)$ are in Mather's nice dimensions, with equality if and only if $f$ is quasihomogeneous.

\section{Preliminaries}
	We refer to \cite{MondNuno_Singularities} for further definitions and proofs of the results shown here.
Throughout this article, we shall use the following notation: we set $\ms{O}_n$ as the ring of germs of functions on $(\mb{C}^n,S)$, $\theta_n$, as the $\ms{O}_n$-module of germs of vector fields $\xi$ on $(\mb{C}^n,S)$, and $\mf{m}_n$, as the ideal of $f \in \ms{O}_n$ such that $f(S)=0$.
Furthermore, for $f\colon (\mb{C}^n,S) \to (\mb{C}^{n+1},0)$, we set $\theta(f)$ as the $\ms{O}_n$-module of germs of vector fields $\xi$ along $f$, and $\Sigma(f)$ as the set-germ of non-immersive points of $f$.
Unless otherwise stated, all maps are assumed holomorphic.

We say $f,g\colon (\mb{C}^n,S) \to (\mb{C}^{n+1},0)$ are \textbf{$\ms{A}$-equivalent} if there exist germs of diffeomorphisms $\psi$, $\phi$ such that $g=\psi\circ f\circ \phi^{-1}$.
A ($d$-parameter) \textbf{unfolding} of $f\colon (\mb{C}^n,S) \to (\mb{C}^{n+1},0)$ is a holomorphic 
\begin{function}
	F\colon (\mb{C}^n\times\mb{C}^d,S\times\{0\}) \arrow[r] & (\mb{C}^{n+1}\times\mb{C}^d,0)\\
	(x,u) \arrow[r, maps to]								& (f_u(x),u)
\end{function}
such that $f_0=f$.
We say $f$ is \textbf{$\ms{A}$-stable} if every $d$-parameter unfolding $F$ of $f$ is $\ms{A}$-equivalent to $f\times\id_{(\mb{C}^d,0)}$.

\begin{definition}
	The \textbf{$\ms{A}$-tangent space} of $f\colon (\mb{C}^n,S) \to (\mb{C}^{n+1},0)$ is the $\mb{C}$-vector space
		\[T\ms{A}_ef=\left\{\left.\frac{df_s}{ds}\right|_{s=0}: f_s=\psi_s\circ f\circ \phi_s\right\} \subseteq \theta(f)\]
	where $\psi_s$ and $\phi_s$ are smooth families of diffeomorphisms.
	We define the \textbf{$\ms{A}$-codimension} of $f$ as the $\mb{C}$-codimension of $T\ms{A}_ef$ in $\theta(f)$.
	We shall say that $f$ is \textbf{$\ms{A}$-finite} if $\codim_{\ms{A}_e}(f) < \infty$.
\end{definition}

Using the chain rule, it is then easy to see that
	\[T\ms{A}_ef=tf(\theta_n)+\omega f(\theta_{n+1})\]
where $tf\colon \theta_n \to \theta(f)$ and $\omega f\colon \theta_{n+1} \to \theta(f)$ are respectively given by $tf(\xi)=df\circ \xi$ and $\omega f(\eta)=\omega\circ f$.

\begin{theorem}
	A holomorphic $f\colon (\mb{C}^n,S) \to (\mb{C}^{n+1},0)$ is $\ms{A}$-stable if and only if $\codim_{\ms{A}_e}(f)=0$.
\end{theorem}

A map-germ $f\colon (\mb{C}^n,S) \to (\mb{C}^{n+1},0)$ is \textbf{finite} if there exists a closed representative $f\colon N \to Z$ with finite fibres. Finite map-germs are an important tool in analytic geometry, since they preserve coherent sheaves (see e.g. \cite{SteinSpaces}).

\begin{theorem}\label{thm: finiteness criterion}
    Given a holomorphic $f\colon(\mb{C}^n,S) \to (\mb{C}^{n+1},0)$, the following statements are equivalent:
    \begin{enumerate}
        \item $f$ is finite;
        \item $\ms{O}_n$ is finitely generated over $(f_1,\dots,f_{n+1})$;
        \item $\dim_\mb{C}\ms{O}_n/(f_1,\dots,f_{n+1})<\infty$;
        \item as germs at $S$, $f^{-1}(\{0\})=S$.
    \end{enumerate}
\end{theorem}

We then have the following geometric criterion for $\ms{A}$-finiteness: $f\colon (\mb{C}^2,S) \to (\mb{C}^3,0)$ is $\ms{A}$-finite if and only if it has a finite representative $f\colon N \to Z$ such that $f(N)$ only contains transversal double points outside the origin.

Another important tool which we shall use in \S 4 are the Fitting ideals of a map-germ.
Let $M$ be a finitely generated module on a Noetherian ring $R$.
A \textbf{finite presentation} on $M$ is an exact sequence in the form
\begin{diagram}
    R^b \arrow[r, "\lambda"] & R^a \arrow[r] & M \arrow[r] & 0
\end{diagram}
The matrix $\lambda$ is known as the presentation matrix of $M$.
We then define the $k$-th \textbf{Fitting ideal} of $M$ as the ideal $\Fitt_k(M)$ in $R$ generated by the $(c-k)\times(c-k)$ minors of $\lambda$, $c=\min\{a,b\}$. For $k > c$, we set $\Fitt_k(M)=0$, and for $k < 0$, $\Fitt_k(M)=R$.
In particular, $\Fitt_k(M)$ is independent of the choice of $\lambda$ (\cite{Arnold_Fitting}), making this notation unambiguous.

Let $f\colon (X,S) \to (Y,y)$ be a finite holomorphic map-germ between complex manifolds, $\ms{O}_X$ be the sheaf of holomorphic functions on $X$, and $f_*\ms{O}_X$ be the pushforward sheaf on $Y$,
    \[f_*\ms{O}_X(U)=\ms{O}_X(f^{-1}(U))\]
By Theorem \ref{thm: finiteness criterion}, $(f_*\ms{O}_X)_y=\oplus\{f_*(\ms{O}_X,x)\colon x \in S\}$ is finitely generated over $\ms{O}_{Y,y}$. Since $\ms{O}_{Y,y}$ is a Noetherian ring, we can consider a finite resolution
\begin{diagram}
    \ms{O}_{Y,y}^b \arrow[r, "\lambda"] & \ms{O}_{Y,y}^a \arrow[r] & (f_*\ms{O}_X)_y \arrow[r] & 0
\end{diagram}
An algorithm to compute the presentation matrix for polynomial maps can be found at \cite{Guille_presmatrix}.

\begin{definition}
    We define the $k$-th Fitting ideal of $f$ at $y$ as the $\ms{O}_{Y,y}$-module
        \[\mc{F}_k(f)=\Fitt_k((f_*\ms{O}_X)_y)\]
    and the $k$-th target multiple point space of $f$ as the zero locus $M_k(f)$ of $\mc{F}_{k-1}(f)$.
\end{definition}

Given $p,q \in \ms{O}_n[z]$, we set the $\ms{O}_n$-module $M_z(p,q)=\ms{O}_{n}[z]/(p,q)$. The $0$th Fitting ideal of this algebra is generated by the algebraic resultant $\Res_z(p,q)$ (see \cite{Teissier_Invariants}), and is thus known as the \textbf{resultant ideal} of $p,q$.

\section{Frontal map-germs}
	Roughly speaking, a frontal hypersurface is an analytic hypersurface $X \subset \mb{C}^{n+1}$ that has a well-defined tangent space at each point.
If $X$ has a singularity at a point $x$, we consider a sequence of regular points $(x_m) \subseteq X$ converging to $x$.

Let $PT^*\mb{C}^{n+1}$ be the projectivized cotangent bundle of $\mb{C}^{n+1}$.
If $(z,[\omega]) \in PT^*\mb{C}^{n+1}$, we equip $PT^*\mb{C}^{n+1}$ with the contact structure given by the differential form
	\[\alpha=\omega_1\,dz^1+\dots+\omega_{n+1}\,dz^{n+1}\]
A holomorphic $F\colon N \subset \mb{C}^n \to PT^*\mb{C}^{n+1}$ is \textbf{integral} if $F^*\alpha=0$.
We also say a projection $\pi\colon PT^*\mb{C}^{n+1} \to \mb{C}^{n+1}$ is a \textbf{Legendrian fibration} for $\alpha$ if $\ker d\pi_{(z,[\omega])} \subseteq \ker\alpha_{(z,[\omega])}$ for all $(z,[\omega])\in PT^*\mb{C}^{n+1}$.

\begin{definition}
	A holomorphic $f\colon N \subset \mb{C}^n \to \mb{C}^{n+1}$ is \textbf{frontal} if there exists an integral $F\colon N \to PT^*\mb{C}^{n+1}$ and a Legendrian projection $\pi\colon PT^*\mb{C}^{n+1} \to \mb{C}^{n+1}$ such that
		\[f=\pi \circ F\]
	If $F$ is an immersion, we say $f$ is a \textbf{wavefront}.
	Similarly, a hypersurface $X \subset \mb{C}^{n+1}$ is frontal (resp. a front) if there exists a frontal (resp. wavefront) $f\colon N \to \mb{C}^{n+1}$ such that $X=f(N)$.
\end{definition}

Let $(z,,p)$ be coordinates for $PT^*\mb{C}^{n+1}$, and $\pi\colon PT^*\mb{C}^{n+1} \to \mb{C}^{n+1}$ be the canonical projection, $\pi(z,p)=z$.
It is easy to see that $\pi$ is a Legendrian fibration for $\alpha$.
If $F\colon N \to PT^*\mb{C}^{n+1}$ is an integral map and $f=\pi\circ F$,
	\[0=F^*\alpha=\sum^{n+1}_{i=1}\nu_i d(Z_i\circ F)=\sum^{n+1}_{i=1}\sum^n_{j=1}\nu_i\frac{\p f_i}{\p x_j}\,dx^j\]
for some $\nu_1,\dots,\nu_{n+1} \in \ms{O}_n$, not all of them zero (since they come from a projective space).
This is the same as claiming that there exists a nowhere-vanishing $\nu\colon N \to T^*\mb{C}^{n+1}$ such that, for all vector fields $\xi$ on $N$,
	\[\nu(df\circ \xi)=0\]

Since $PT^*\mb{C}^{n+1}$ is a fibre bundle, we can find for each pair $(z,[\omega]) \in PT^*\mb{C}^{n+1}$ an open neighbourhood $Z \subset \mb{C}^{n+1}$ of $z$ and an open $U \subseteq \mb{C}P^{n+1}$ such that $\pi^{-1}(Z)\cong Z\times U$.
Therefore, the integral map $F$ is $\ms{A}$-equivalent to $\tilde f(x)=(f(x),[\nu_x])$, known as the \textbf{Nash lift} of $f$.

We shall denote the family of differential $1$-forms $\nu$ along $f$ that vanish nowhere as $\Omega^1_f(U)$, and the family of germs of those as $\Omega^1_f(\mb{C}^n,0)$.

\begin{example}
	\begin{enumerate}
		\item Every analytic plane curve is frontal: given an analytic $\gamma(x)=(p(x),q(x))$, let $k=\min\{\ord p,\ord q\}$.
		The vector field $\gamma'$ is orthogonal to the nowhere-vanishing differential form
			\[\nu_x=\frac{1}{x^k}(q'(x)\,dX-p'(x)\,dY)\]

		\item The folded Whitney umbrella can be parametrized as
		\begin{function}
			f\colon (\mb{C}^2,0) \arrow[r] & (\mb{C}^3,0)\\
			(x,y) \arrow[r, maps to] & (x,y^2,xy^3)
		\end{function}
		If $\star\colon\Omega^2(f) \to \Omega^1(f)$ is the Hodge dual over germs of differential forms along $f$ and $(X,Y,Z)$ are coordinates for $\mb{C}^3$,
			\[\star\left(\frac{\p f}{\p x}\wedge\frac{\p f}{\p y}\right)=-2y^4\,dX-3xy^2\,dY+2y\,dZ=-y\nu_{(x,y)}\]
		and thus $\nu_{(x,y)}$ is a germ of $1$-form along $f$ that vanishes nowhere.
	\end{enumerate}
\end{example}

\begin{example}\label{ex: f4 mond}
	The $F_4$ singularity is described by Mond \cite{Mond_Classification} as the map-germ $f\colon (\mb{C}^2,0) \to (\mb{C}^3,0)$ given by
		\[f(x,y)=(x,y^2,y^5+x^3y)\]
	We wish to know whether $f$ is a frontal germ.

	Let $\nu\colon (\mb{C}^2,0) \to T^*\mb{C}^3$ be a germ of $1$-form along $f$ such that $\nu(df\circ \xi)=0$ for all $\xi \in \theta_2$.
	Since $\im df$ has dimension $2$ almost everywhere, there exists a $\mu\in \ms{O}_2$ such that
		\[\mu(x,y)\nu_{(x,y)}=\star\left(\frac{\p f}{\p x}\wedge\frac{\p f}{\p y}\right)(x,y)=2y^2\, dX + (5y^4+x^3)\, dY-2y\,dZ\]
	
	Assume there is a $\nu\in \Omega^1_f(\mb{C}^2)$ such that $\nu(df\circ \xi)=0$ for all vector fields $\xi$ on $U$.
	If $\nu=\nu_1\,dX+\nu_2\,dY+\nu_3\,dZ$,
	\begin{align*}
		\mu\nu_1=2y^2;&& \mu\nu_2=5y^4+x^3;&& \mu\nu_3&=-2y
	\end{align*}
	for some function $\mu$.
	However, it is clear that $\mu$ is a unit, so $\nu$ vanishes at $0$ and $f$ is not a frontal.
\end{example}

During the rest of this article, we shall assume all frontals are generically immersive.

\begin{definition}
	A smooth multigerm $f\colon (\mb{C}^n,S) \to \mb{C}^{n+1}$ is \textbf{frontal} if it has a frontal representative $f\colon N \to Z$.
	Given a hypersurface $X \subset \mb{C}^{n+1}$, $(X,0)$ is a \textbf{frontal} if there exists a frontal $f\colon (\mb{C}^n,S) \to (\mb{C}^{n+1},0)$ such that $(X,0)=f(\mb{C}^n,S)$.
\end{definition}

We now give a characterization for corank $1$ frontal map-germs:
\begin{proposition}[\cite{Nuno_CuspsAndNodes}]\label{prop: corank 1 frontal}
	Let $f\colon (\mb{K}^n,0) \to (\mb{K}^{n+1},0)$ be a corank $1$ mono-germ.
	Choose local coordinates $(x_1,\dots,x_{n-1},y)$ in the source and $(Y_1,\dots,Y_{n+1})$ in the target such that
	\begin{equation}\label{eq: prenormal corank 1}
		f(x,y)=(x,p(x,y),q(x,y))
	\end{equation}
	for some $p,q \in \ms{O}_n$.
	Then $f$ is frontal if and only if $p_y|q_y$ or $q_y|p_y$.
\end{proposition}

We shall say that $f$ is in \textbf{prenormal form} if it is given as in Equation \ref{eq: prenormal corank 1} with $p_y|q_y$, in which case we set $\mu=q_y/p_y$.
We now state a series of definitions and results that we shall use throughout this paper.
Proofs for these statements can be found in \cite{Frontals}.

\begin{proposition}
	Two frontal multi-germs $f,g\colon (\mb{C}^n,S) \to (\mb{C}^{n+1},0)$ are $\ms{A}$-equivalent if and only if their Nash lifts $\tilde f, \tilde g$ are Legendrian equivalent.
\end{proposition}

\begin{corollary}\label{cor: frontals preserved under A}
	Let $f,g\colon (\mb{C}^n,S) \to \mb{C}^{n+1}$ be smooth germs.
	If $f$ is $\ms{A}$-equivalent to $g$ and $f$ is frontal, $g$ is frontal.
\end{corollary}

Given a frontal $f\colon (\mb{C}^n,S) \to (\mb{C}^{n+1},0)$, we define the space of infinitesimal frontal deformations of $f$ as
	\[\ms{F}(f):=\left\{\left.\frac{df_t}{dt}\right|_{t=0}: f_0=f, (f_t,t) \text{ frontal}\right\}\]
Using Corollary \ref{cor: frontals preserved under A}, we see that the $\ms{A}$-orbit of $f$ is contained within the space of frontal mappings $(\mb{C}^n,S) \to (\mb{C}^{n+1},0)$.
It is then clear that $T\ms{A}_ef \subseteq \ms{F}(f)$, and thus we can assign a frontal codimension to $f$ by computing the codimension of $T\ms{A}_ef$ in $\ms{F}(f)$:

\begin{definition}
	Let $f\colon (\mb{C}^n,S) \to (\mb{C}^{n+1},0)$ be a frontal germ.
	We define the \textbf{frontal codimension} or $\ms{F}$-codimension of $f$ as
		\[\codim_{\ms{F}e}(f):=\dim_\mb{C}\frac{\ms{F}(f)}{T\ms{A}_ef}\]
	We say $f$ is $\ms{F}$-finite if $\codim_{\ms{F}_e}(f) < \infty$.
\end{definition}

Given a smooth $f\colon (\mb{C}^n,S) \to (\mb{C}^{n+1},0)$ (not necessarily frontal), we shall say $f$ has \textbf{corank $1$} if there exists a representative $f\colon N \to Z$ of $f$ such that the dimension of $\ker df_x$ is less than or equal to $1$ for all $x \in S$.
Note that this definition includes germs with immersive branches, such as the transverse intersection of $n$ hyperplanes.

\begin{definition}
	Let $f\colon (\mb{C}^n,S) \to (\mb{C}^{n+1},0)$ be a frontal multi-germ.
	We say $f$ is stable as a frontal or $\ms{F}$-stable if every $d$-parameter frontal unfolding $F$ of $f$ is $\ms{A}$-equivalent to $f\times\id_{(\mb{C}^d,0)}$.	
\end{definition}

\begin{theorem}
	Let $f\colon (\mb{C}^n,0) \to (\mb{C}^{n+1},0)$ be a generically immersive frontal germ of corank $1$.
	Then $f$ is stable as a frontal if and only if $\ms{F}(f)=T\ms{A}_ef$.
\end{theorem}

Given a frontal $f\colon (\mb{C}^n,S)\to (\mb{C}^{n+1},0)$, we define the vector subspace $\hat\tau(f) \subseteq T_0\mb{C}^{n+1}$ as follows:
	\[\hat\tau(f)=\mathrm{ev}\left\{\omega f^{-1}[tf(\theta_{n,S})+(f^*\mfm_{n+1,0})\ms{F}(f)]\right\}\]
where $\mathrm{ev}\colon \theta_{n+1} \to T_0\mb{C}^{n+1}$ denotes evaluation at $0$.

\begin{proposition}
	A frontal multi-germ $f\colon (\mb{C}^n,S) \to (\mb{C}^{n+1},0)$ with branches $f_1,\dots,f_r$ is stable as a frontal if and only if $f_1,\dots,f_r$ are stable as frontals and the vector subspaces $\hat\tau(f_1),\dots,\hat\tau(f_r)\subseteq T_0\mb{K}^{n+1}$ meet in general position.
\end{proposition}

We now state an adaptation of the Mather-Gaffney criterion for frontal germs:
\begin{theorem}
\label{thm: frontal mather gaffney}
	Let $f\colon (\mb{C}^n,S) \to (\mb{C}^{n+1},0)$ be a corank $1$ frontal map-germ.
	If $n=1$ or $V(p_y,\mu_y)=S$ as set-germs, $f$ is $\ms{F}$-finite if and only if there exists a representative $f\colon N' \to Z'$ of $f$ such that
	\begin{enumerate}
		\item $f^{-1}(0)\subseteq S$;
		\item the restriction $f\colon N' \backslash \{0\} \to Z'\backslash \{0\}$ is locally stable as a frontal.
	\end{enumerate}
\end{theorem}

In particular, if $f$ is a surface, we will later show (Proposition \ref{prop: finite frontal type}) that $f$ can only contain cuspidal edges and transversal double points outside $S$.

\subsection{Frontalisation of a fold surface}
Mond classified the $\ms{A}$-simple monogerms $(\mb{R}^2,0) \to (\mb{R}^3,0)$ in \cite{Mond_Classification}.
The resulting classification shows that all but one of the families (labelled $H_k$) is fold-type.
A surface $f\colon N^2 \to Z^3$ is fold-type if we can find local coordinates $(x,y)$ for $N$ and $(X,Y,Z)$ for $Z$ such that
	\[f(x,y)=(x,y^2,yp(x,y^2))\]
for some $p\colon N \to \mb{R}$.

Using the characterisation of corank $1$ frontal maps given in Proposition \ref{prop: corank 1 frontal}, it follows that $f$ is frontal if and only if $p(x,y)=yq(x,y)$, in which case $f(x,y)=(x,y^2,y^3q(x,y^2))$.
With this observation in mind, we can \emph{frontalise} the fold families from Mond's classification, leading to the results shown in Table \ref{table: mond simple} and Figure \ref{fig: mond simple}.

\begin{table}[ht]
\centering
\begin{tabular}{LLLL|CCC}
	\hline
	\multicolumn{2}{c}{Mond's classification}		& \multicolumn{2}{c}{Frontalised surface}	& \text{Codimension}	& \text{Notes}	\\
	\hline
	\hline
	\T S_k & y^3+x^{k+1}y							& \B \check{S}_k & y^5+x^{k+1}y^3			& k 					& 				\\
	\T B_k & x^2y+y^{2k+1}							& \B \check{B}_k & x^2y^3+y^{2k+3}			& k						& k \geq 2		\\
	\T C_k & xy^3+x^ky								& \B \check{C}_k & xy^5+x^ky^3				& k						& k \geq 3		\\
	\T F_4 & y^5+x^3y								& \B \check{F}_4 & y^7+x^3y^3 				& 4						&				
\end{tabular}
\caption{Simple fold surfaces from \cite{Mond_Classification} along with their frontal counterparts.
The $\ms{A}$-codimension of the former coincides with the $\ms{F}$-codimension of the latter.\label{table: mond simple}}
\end{table}

\begin{figure}[ht]
\centering
     \begin{subfigure}[b]{\textwidth}
         \centering
		\includegraphics[width=0.2\textwidth]{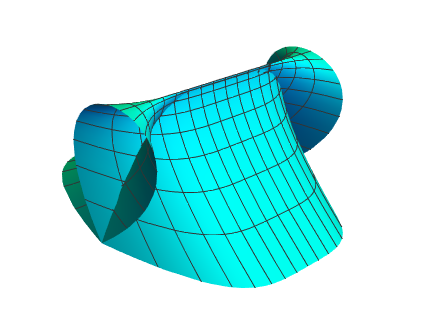}
		\includegraphics[width=0.2\textwidth]{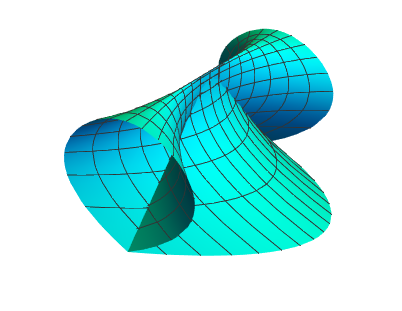}
		\includegraphics[width=0.2\textwidth]{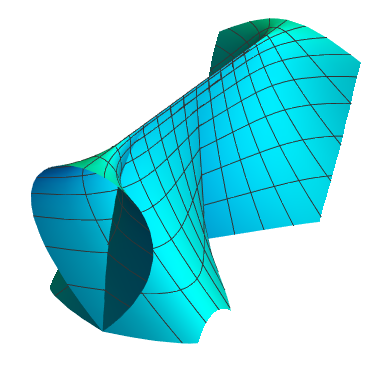}
		\includegraphics[width=0.2\textwidth]{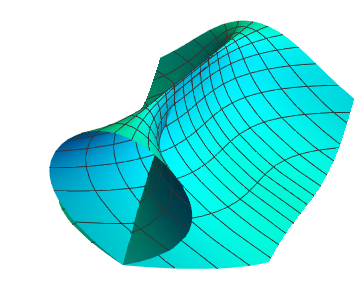}
         \caption{$S_k$, $B_k$, $C_k$, $F_4$.}
     \end{subfigure}\\
     \begin{subfigure}[b]{\textwidth}
         \centering
		\includegraphics[width=0.2\textwidth]{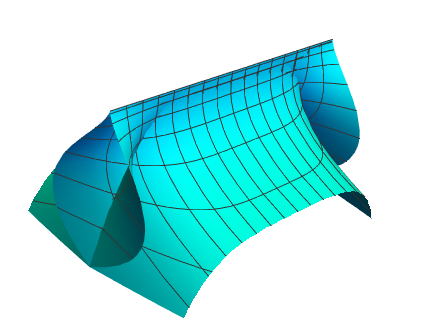}
		\includegraphics[width=0.2\textwidth]{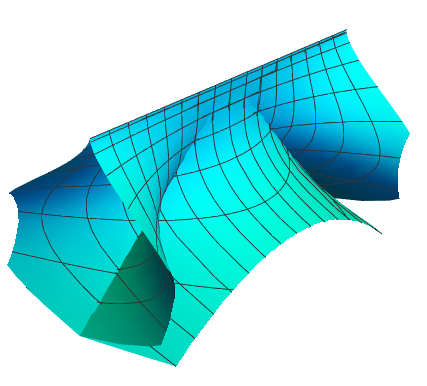}
		\includegraphics[width=0.2\textwidth]{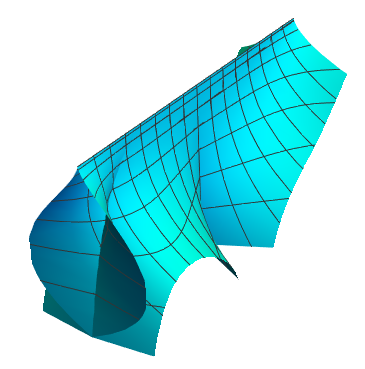}
		\includegraphics[width=0.2\textwidth]{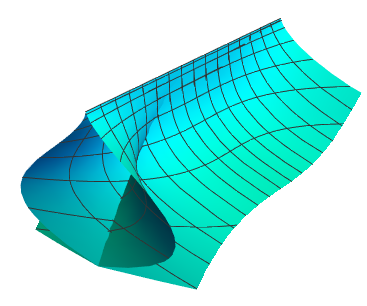}
         \caption{$\check{S}_k$, $\check{B}_k$, $\check{C}_k$, $\check{F}_4$.}
     \end{subfigure}
\caption{Fold surfaces from Mond's classification (top), along with their frontalised counterpart (bottom). \label{fig: mond simple}}
\end{figure}

We now define a precise definition of frontalisation.
\begin{definition}
	Let $f\colon (\mb{C}^2,S) \to (\mb{C}^3,0)$ be given by $f(x,y)=(x,y^2,yp(x,y^2))$.
	We define the \textbf{frontalisation} of $f$ as the fold surface $\check{f}\colon (\mb{C}^2,S) \to (\mb{C}^3,0)$ given by
		\[\check{f}(x,y)=(x,y^2,y^3p(x,y^2))\]
\end{definition}

Note that none of the frontal surfaces in Table \label{table: mond simple} is a wavefront.
More generally, if $\check{f}$ is a wavefront, the function
    \[\cvf{y}\left(y^3p(x,y^2)\right)=3y^2p(x,y^2)+2y^4p_y(x,y^2)\]
has to be in the form $\alpha(x)+\beta y^2$ for some $\alpha \in \ms{O}_1$, $\beta \in \mb{C}$ (so that its Nash lift is an immersion).
However, this can only happen if $p_y(x,y)=0$, in which case $\check{f}$ is either an immersion or a cuspidal edge.

The goal of this section is to prove the following
\begin{theorem}\label{thm: frontalisation preserves codimension}
	Given a fold surface $f\colon (\mb{C}^2,0) \to (\mb{C}^3,0)$ with frontalisation $\check{f}$,
		\[\frac{\theta(f)}{T\ms{A}_ef}\cong \frac{\ms{F}(\check{f})}{T\ms{A}_e\check{f}}\]
	In particular, $f$ is $\ms{A}$-finite if and only if $\check{f}$ is $\ms{F}$-finite and $\codim_{\ms{A}}(f)=\codim_{\ms{F}}(\check{f})$.
\end{theorem}

In order to prove this result, we consider a $\mb{C}$-linear mapping $K\colon \theta(f) \to \ms{F}(\check{f})$ that frontalises the infinitesimal deformations of $f$.
This map will be an epimorphism by construction and will send $T\ms{A}_ef$ onto $T\ms{A}_e\check{f}$, thus inducing the desired isomorphism between the quotient spaces.

Let $i\colon (\mb{C},0) \to (\mb{C}^2,0)$ be the germ of immersion given by $i(x)=(x,0)$.
If $\xi=(\xi_1,\xi_2,\xi_3) \in \theta(f)$, we can find a unique $\hat{\xi}_3$ such that $\xi=(\xi_1,\xi_2,i^*\xi_3)+(0,0,y\hat{\xi}_3)$, thus yielding the following decomposition of $\theta(f)$:
	\[\theta(f)=(\ms{O}_2^2\times i^*\ms{O}_2)\oplus\ms{O}_2\{(0,0,y)\}\]
We then define the linear map $K$ as follows:
\begin{enumerate}
	\item if $\xi \in \ms{O}_2^2\times i^*\ms{O}_2$, $K(\xi)=\xi$;
	\item if $\xi \in \ms{O}_2\{(0,0,y)\}$, $K(\xi)=y^2\xi$.
\end{enumerate}

As shown in \cite{Frontals}, $\xi \in \ms{F}(\check{f})$ if and only if there exist $\alpha,\beta \in \ms{O}_2$ such that $d\xi_3=\alpha\,dx+y\beta\,dy$.
It then follows that
	\[\ms{F}(\check{f})\cong \frac{\theta(\check{f})}{\ms{O}_2\{(0,0,y)\}}\]
and $K$ is a $\mb{C}$-linear epimorphism.

\begin{lemma}\label{lemma: K preserva tangentes}
	Given a fold surface $f$, $K(T\ms{A}_ef)=T\ms{A}_e\check{f}$.
\end{lemma}

Throughout the proof of this lemma, we shall use the following notation, borrowed from \cite{Mond_Classification}:
\begin{enumerate}
	\item $\ms{O}_2^T:=\{g(x,y^2): g \in \ms{O}_2\}$;
	\item $\mf{m}_2^T:=\{g(x,y^2): g \in \mf{m}_2\}=\ms{O}_2^T\cap\mf{m}_2$;
	\item if $h(x,y)=(x,y^2)$, $T_e\ms{K}^T(p\circ h):=h^*(p_x,yp_y,p)$.
\end{enumerate}

\begin{proof}[Proof of Lemma \ref{lemma: K preserva tangentes}]
	Let $f(x,y)=(x,y^2,yp(x,y^2))$, and $q(x,y)=yp(x,y)$.
	By definition of $T_e\ms{K}^T$, we have
		\begin{align*}
			T_e\ms{K}^T(q\circ h)	&=h^*(q_x,yq_y,q)=h^*(yp_x,y(p+yp_y),yp)=\\
									&=y^2h^*(p_x,yp_y,p)=y^2T_e\ms{K}^T(p\circ h)
		\end{align*}

	Using \cite{Mond_Classification}, we then have
	\begin{align*}
		T\ms{A}_ef=			&\ms{O}_2\cvf{X}\oplus\ms{O}_2\cvf{Y}\oplus(\ms{O}_2^T+yT_e\ms{K}^T(p\circ h))\cvf{Z};\\
		T\ms{A}_e\check{f}=	&\ms{O}_2\cvf{X}\oplus\ms{O}_2\cvf{Y}\oplus(\ms{O}_2^T+yT_e\ms{K}^T(q\circ h))\cvf{Z}
	\end{align*}

	Given $\xi \in T\ms{A}_ef$, there exist $\xi_1,\xi_2 \in \ms{O}_2$, $\xi_3 \in \ms{O}_2^T$ and $\xi_4 \in T_e\ms{K}^T(p\circ h)$ such that
		\[\xi=\xi_1\cvf{X}+\xi_2\cvf{Y}+(\xi_3+y\xi_4)\cvf{Z}\]
	In addition, there exists a unique $\hat{\xi_3} \in \ms{O}_2$ such that $\xi_3=i^*\xi_3+y\hat{\xi_3}$.
	We then have
		\[K(\xi)=\xi_1\cvf{X}+\xi_2\cvf{Y}+(i^*\xi_3+y^3\hat{\xi_3}+y^3\xi_4)\cvf{Z} \in T\ms{A}_e\check{f}\]
	and thus $K(T\ms{A}_ef)\subseteq T\ms{A}_e\check{f}$.

	Conversely, let $\eta \in T\ms{A}_e\check{f}$: there exist $\eta_1,\eta_2 \in \ms{O}_2$, $\eta_3 \in \ms{O}_2^T$ and $\eta_4 \in T_e\ms{K}^T(p\circ h)$ such that
		\[\eta=\eta_1\cvf{X}+\eta_2\cvf{Y}+(\eta_3+y^3\eta_4)\cvf{Z}\]
	In addition, there exists a unique $\hat{\eta_3} \in \ms{O}_2$ such that $\eta_3=i^*\eta_3+y^2\hat{\eta_3}$.
	We then have
		\[\eta=K\left(\eta_1\cvf{X}+\eta_2\cvf{Y}+(i^*\eta_3+\hat{\xi_3}+y\xi_4)\cvf{Z}\right) \in K(T\ms{A}_ef)\]
	proving the opposite inclusion.
\end{proof}

Since $K\colon \theta(f) \to \ms{F}(\check{f})$ is surjective and $K(T\ms{A}_ef)=T\ms{A}_e\check{f}$, the induced map
	\[\hat K\colon \frac{\theta(f)}{T\ms{A}_ef} \to \frac{\ms{F}(\check{f})}{T\ms{A}_e\check{f}}\]
is an isomorphism, thus proving Theorem \ref{thm: frontalisation preserves codimension}.

\section{Double point curve of a frontal surface}\label{doublepoints}
	Marar and Tari \cite{MararTari} studied the geometric invariants of corank $1$ smooth maps $F\colon (\mb{R}^3,0) \to (\mb{R}^3,0)$.
One of their findings was that the source double point space (see \S 5.2 below) of the discriminant of $F$ is given by an equation in the form $dc^2=0$, where $c,d\colon (\mb{R}^2,0) \to \mb{R}$ are the equations for the cuspidal edge and transverse double point curves in the discriminant of $F$.
In particular,  the discriminant of $F$ is a front if and only if $\Sigma(F)$ is smooth in $(\mb{R}^3,0)$ (see \cite{Arnold_CWF}).
Therefore, this formula can be also applied to $f=F|_{\Sigma(F)}$, seen as a wavefront.

In this section, we show that this formula holds for any frontal, not just fronts.

\subsection{Pre-normal form of a corank $1$ injective map}
Let $f\colon (\mb{C}^2,0) \to (\mb{C}^3,0)$ be a corank $1$ frontal map-germ in prenormal form.
We define $D^2(f)$ as the space in $(\mb{C}^3,0)$ given by the equations
	\[\frac{p(x,y)-p(x,y')}{y-y'}=\frac{q(x,y)-q(x,y')}{y-y'}=0\]
We then have $D^2(f)=C^2(f)\cup D_+^2(f)$, where
\begin{align*}
	C^2(f)=\{(x,y,y): (x,y) \in \Sigma(f)\}; && D_+^2(f)=\overline{\{(x,y,y') \in D^2(f): y\neq y'\}}
\end{align*}
Note that $D_+^2(f)$ coincides with the double point curve for $\ms{A}$-finite surfaces.

Let us first assume $f$ is injective (hence finite by the Nullstellensatz) and let $\pi\colon D^2(f) \to (\mb{C}^2,0)$ be given by $(x,y,y') \mapsto (x,y)$.
Since $f$ is finite, $\pi$ is also finite and $D(f)=\pi(D^2(f))$ is an analytic set by Remmert's Proper Mapping Theorem (\cite{Lojasiewicz_Remmert}).

\begin{definition}
	We define the \textbf{double point space} of $f$ as $D(f)$.
\end{definition}

Since $\pi\colon D^2(f) \to \mb{C}^2$ is a finite germ, $\ms{O}_{D^2(f)}$ admits a finite presentation in the form
\begin{diagram}
	\ms{O}_2^r \arrow[r,"\lambda"] & \ms{O}_2^s \arrow[r] & \ms{O}_{D^2(f)} \arrow[r] & 0
\end{diagram}

Since $f$ is assumed to be injective, $D(f)=\Sigma(f)$.
Let $\Delta(f)=f(D(f))$.
Since $f$ is complex analytic, $\Delta(f)$ has codimension $1$, so its dimension must be exactly $1$.
We can then choose a generic plane $H_0 \subset \mb{C}^3$ passing through $0$ such that $H_0\cap \Sigma=\{0\}$ and $f$ is transverse to $H_0$ at $0$.

If $(X,Y,Z)$ are coordinates for $(\mb{C}^3,0)$, we can assume $H_0$ is the plane of equation $X=0$, so a suitable change of coordinates transforms $f$ into a mapping of the form
	\[f(x,y)=(x,\tilde p(x,y),\tilde q(x,y))\]
where $\tilde p, \tilde q$ have order greater than $1$.
However, since $D^2(f)$ is preserved by diffeomorphisms, we can assume $\tilde p=p$ and $\tilde q=q$.
We then define $L_0=f^{-1}(H_0)$, which is the line of equation $x=0$.

\begin{lemma}\label{lemma: local intersection number}
	Suppose that we have a curve $(X,0)$ and a hypersurface $(Y,0)$ in $(\mb{C}^n,0)$ such that $X\cap Y=\{0\}$.
	Then, $i_0(X,Y)=1$ if and only if $(X,0)$, $(Y,0)$ are smooth submanifold-germs of $(\mb{C}^n,0)$ and $X\pitchfork Y$.
\end{lemma}

\begin{proof}
	Let $(X_1,0),\dots,(X_r,0)$ be the irreducible components of $(X,0)$.
	Since the local intersection number is additive,
		\[i_0(X,Y)=\sum^r_{i=1}i_0(X_i,Y)\]

	For each $i=1,\dots,r$, we can consider a holomorphic $\gamma_i\colon (\mb{C},0) \to (\mb{C}^n,0)$ which is finite, generically $1$-to-$1$ and whose image is $(X_i,0)$.
	If $g\colon (\mb{C}^n,0) \to (\mb{C},0)$ is the reduced equation for $(Y,0)$, $g\circ \gamma_i$ is not constant, so we can write
		\[(g\circ \gamma_i)(u)=a_ku^k+a_{k+1}u^{k+1}+\dots\]
	with $k=\ord_0(g\circ \gamma_i)$.
	For $t\neq 0$ and close to the origin, $t$ is a regular value of $g\circ \gamma_i$ and has exactly $k$ preimages in a neighbourhood of $0$.
	Therefore, $i_0(X_i,Y)=\ord_0(g\circ \gamma_i)$.

	Suppose now that $i_0(X,Y)=1$.
	Necessarily $r=1$ and $\ord(g\circ \gamma_1)=1$, so $g\circ \gamma_1$ is a diffeomorphism.
	This means $\gamma_1$ is an immersion and $g$ is a submersion.
	Thus, $(X,0)$ and $(Y,0)$ are both smooth and $X\pitchfork Y$.
\end{proof}

\begin{lemma}\label{lemma: delta smooth}
	With the above notation we have:
	\begin{enumerate}
		\item $\Delta(f)$ is smooth and transverse to $H_0$ at the origin.
		\item $D(f)$ is smooth and transverse to $L_0$ at the origin.
	\end{enumerate}
\end{lemma}

\begin{proof}
	Let $H_t$ be the plane in $\mb{C}^3$ given by $x=t$.
	We consider a small enough representative of the form $f\colon U \to T\times V$, where $U,T \subset \mb{C}$ and $V\subset \mb{C}^2$ are open neigbourhoods of the origin, such that
	\begin{enumerate}
		\item $f$ is transverse to $H_t$ for all $t \in T$,
		\item $H_0\cap \Delta(f)=\{0\}$,
		\item $H_t$ is transverse to $\Sigma$ for all $t\neq 0$.
	\end{enumerate}

	For each $t \in T$, we define $Y_t$ as the analytic subset of $V$ such that $X\cap H_t=\{t\}\times Y_t$.
	Then $Y_t$ can be parametrised as the curve $\gamma_t \colon U_t \to V$ given by $\gamma_t(y)=(p(t,y),q(t,y))$, where $U_t=\{y \in \mb{C}\colon (t,y) \in U\}$.

	Let $t \in T\backslash \{0\}$ and $x_1,\dots,x_m$ be the singular points of $Y_t$.
	By the conservation of the delta invariant, we have
		\[\delta(Y_0,0)=\sum^m_{i=1}\delta(Y_t,x_i)\]
	Since $f$ is injective, $\gamma_0$ and $\gamma_t$ are also injective.
	Thus, $Y_0$ is irreducible at $0$ and $Y_t$ is also irreducible at each $x_i$.
	By Milnor's formula, we get
		\[\mu(Y_0,0)=\sum^m_{i=1}\mu(Y_i,x_i)\]
	A theorem due independently to Gabrièlov, Lazzeri and Lê\cite{Gabrielov_Multiplicity, Lazzeri_Multiplicity, Le_Multiplicity} states that, in these condition, $m=1$.

	On the other hand, the number $m$ is equal to the local intersection number $i_0(\Sigma;H_0)$.
	Since $H_0$ meets $\Delta(f)$ transversally, it follows from Lemma \ref{lemma: local intersection number} that $\Delta(f)$ is smooth and transverse to $H_0$ at $0$.

	The second item is a consequence of the first one.
\end{proof}

\begin{theorem}[Pre-normal form of a corank $1$ injective map]\label{thm: injective prenormal}
	Let $f\colon (\mb{C}^2,0) \to (\mb{C}^3,0)$ be an injective holomorphic map-germ of corank $1$.
	Then $f$ is a frontal and is $\ms{A}$-equivalent to a map-germ in the form
		\[(x,y) \mapsto (x,y^m,y^mh(x,y))\]
	for some $h \in \ms{O}_2$ and $m \geq 2$.
\end{theorem}

\begin{proof}
	Let $\mu \in \ms{O}_2$ such that $\mu=0$ is the reduced equation of $D(f)$.
	By Lemma \ref{lemma: delta smooth}, $D(f)$ is smooth and transverse to $L_0$, so $\mu_y$ does not vanish at the origin.
	We consider the diffeomorphism $\phi(x,y)=(x,\mu(x,y))$.
	We have that the set of non-immersive points of $f\circ \phi^{-1}$ is the line $y=0$.

	Assume now that $f(x,y)=(x,p(x,y),q(x,y))$ for some $p,q \in \mfm^2_2$.
	If
	\begin{align*}
		p(x,y)=p_0(x)+yp_1(x,y); && q(x,y)=q_0(x)+yq_1(x,y)
	\end{align*}
	we define the diffeomorphism $\psi(X,Y,Z)=(X,Y-p_0(X),Z-q_0(X))$.
	Then,
		\[(\psi\circ f)(x,y)=(x,yp_1(x,y),yq_1(x,y))\]

	Using the notation from Lemma \ref{lemma: delta smooth}, $Y_t$ is parametrised as the curve $\gamma_t(y)=(p_t(y),q_t(y))$, where $p_t(y)=p(t,y)$ and $q_t(y)=q(t,y)$.
	Then $Y_t$ has a unique singular point at $\gamma_t(0)=0$ and $\mu(Y_t,0)= \mu(Y_0,0)$ for all $t \in T$.
	By a result of Zariski \cite{Zariski}, $\{Y_t\}$ also has constant multiplicity at the origin.

	For $t=0$, $\gamma_0$ is injective, so $p_0$ and $q_0$ are not identically $0$.
	Assume then that $m=\ord_0(p_0)$ and $k=\ord_0(q_0)$.
	This implies that $m=m(Y_t,0)$ for all $t \in T$.

	On the other hand, by the Weierstrass Preparation Theorem, we can write $p(x,y)=r(x,y)y^m$ and $q(x,y)=y^mh(x,y)$ as for some unit $r \in \ms{O}_2$ and $s \in \ms{O}_2$.
	In conclusion,
		\[f(x,y) \sim_\ms{A} (x,y^m,y^mh(x,y))\]
	as claimed.
\end{proof}

We now move onto the generically injective case.

\begin{proposition}
	Let $f\colon (\mb{C}^2,0) \to (\mb{C}^3,0)$ be a holomorphic map-germ of corank 1.
	If $f$ is generically injective, there exists a representative $f\colon N \to Z$ of $f$ such that $f$ is locally injective outside $0$.
\end{proposition}

\begin{proof}
	Since $f$ is generically injective, $D^2(f)$ is an analytic subset of dimension $1$.
	Therefore, $D^2(f)$ admits a decomposition into irreducible components in the form
	\begin{align*}
		C^2(f)=C_1\cup\dots\cup C_r; && D_+^2(f)=D_1\cup\dots\cup D_s
	\end{align*}
	Note that every two sets in the family
		\[\ms{I}=\{C_1, \dots, C_r, D_1, \dots, D_s, D^2(f)\cap \Delta\}\]
	are different, so we can use the Curve Selection Lemma (see \cite{CurveSelectionLemma}) to choose closed representatives which only meet at the origin (see Figure \ref{fig: double points intersecting branches}).

\begin{figure}[ht]
\begin{minipage}{0.5\textwidth}
\centering
\begin{tikzpicture}
	\draw (-.7,2) node {$D^2(f)$};

	\clip (0,0) rectangle (4.5,4.5);
	\draw[fill=black!30,draw opacity=0] (0,0) circle (4cm);

	\begin{scope}
	\clip (0,0) circle (4cm);
	\draw[dashed] (0,0) -- (4,4);									
	\draw[dashed] (0,0) .. controls (3,0) and (0,3) .. (1,4);		
	\draw[dashed] (0,0) .. controls (1,3) and (3,1) .. (4,2);		
	\end{scope}

	\begin{scope}
	\clip (0,0) circle (1.75cm);
	\draw (0,0) circle (1.75cm);
	\draw[thick] (0,0) -- (4,4);									
	\draw[thick] (0,0) .. controls (3,0) and (0,3) .. (1,4);		
	\draw[thick] (0,0) .. controls (1,3) and (3,1) .. (4,2);		
	\end{scope}

	\draw (3.1,3.1) node {$\Delta$};
	\draw (1.1,4.1) node {$C_i$};
	\draw (3.9,1.9) node {$D_j$};
\end{tikzpicture}
\end{minipage}
\hfill
\begin{minipage}{0.45\textwidth}
\caption{
	The set of points where two given curves in $\ms{I}$ meet cannot accumulate into the origin.
	Therefore, we can choose a small enough neighbourhood such that no two sets meet outside the origin.\label{fig: double points intersecting branches}
	}
\end{minipage}
\end{figure}
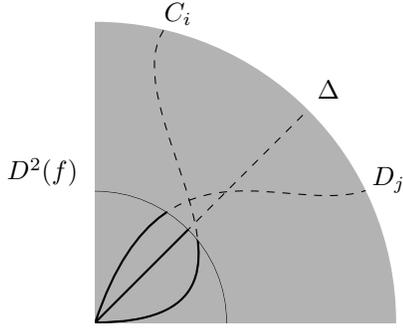

	Let $(x,y,y') \in D^2(f)\backslash\{(0,0,0)\}$.
	\begin{enumerate}
		\item If $y \neq y'$, $(x,y,y') \in D_j$ for some $1\leq j \leq s$.
			We can find an open neighbourhood $V \subset \mb{C}^3$ of $(x,y,y')$ such that $V\cap D^2(f)= V\cap D_j$.
			Since $V$ does not meet $C_1\cup\dots\cup C_r$, $f$ is immersive on $\pi(V)$ and thus locally injective.
		\item If $y=y'$, $(x,y,y) \in C_i$ for some $1 \leq i \leq r$.
			We can find an open neighbourhood $U \subset \mb{C}^3$ such that $U\cap D^2(f)=U\cap C_i$.
			
			Let $U_1,U_2 \subset \mb{C}$ be respective open neighbourhoods of $x$ and $y$ such that $U_1\times U_2\times U_2 \subset U$, and $U_0=U_1\times U_2$.
			If there exist $(x',y'),(x',z') \in U_0$ such that $f(x',y')=f(x',z')$, the point $(x',y',z')$ is in $U\cap D^2(f)$.
			However, we have by construction that
				\[U\cap D^2(f)=U\cap C_i\]
			so $(x',y',z') \in C_i$ and $y'=z'$.
			Therefore, $f|_{U_0}$ is injective.
	\end{enumerate}
\end{proof}

\subsection{Branches of the double point space}
Our goal is to prove that, if $f$ is frontal and $D(f)$ is generated by $\lambda \in \ms{O}_2$, then $p_y^2$ divides $\lambda$.
In order to do so, we shall consider the \textbf{slices} of $f$, $\gamma_t(y)=(p(t,y),q(t,y))$.

Let $\gamma\colon (\mb{C},0) \to (\mb{C}^2,0)$ be a plane curve with isolated singularity at $0$.
We define
	\[D^2(\gamma)=\left\{(t,s) \in \mb{C}^2\colon \frac{\gamma(t)-\gamma(s)}{t-s}=0\right\}\]
Once again, $D^2(\gamma)$ is an analytic set.
Since $\gamma$ has an isolated singularity, it is finite, so the projection-germ $\pi'\colon D^2(\gamma) \to (\mb{C},0)$ given by $\pi'(t,s)=t$ is also finite and $D(\gamma)=\pi'(D^2(\gamma))$ is an analytic set by Remmert's Proper Mapping Theorem.
We consider $D(\gamma)$ with the analytic structure given by the Fitting ideals.

\begin{lemma}\label{lemma: curva milnor divisor}
	Let $\gamma\colon (\mb{C},0) \to (\mb{C}^2,0)$ be a plane curve in the form $\gamma(s)=(s^m,q(s))$ with $q \in \mf{m}_1$.
	If $\gamma$ has an isolated singularity at $0$ and $\mu$ is the Milnor number of $\gamma$, $D(\gamma)$ is the zero locus of the function $d(s)=s^\mu$.
\end{lemma}

\begin{proof}
	Since $\gamma$ has an isolated singularity, it is $\ms{A}$-finite and we can find a stabilisation $\gamma_t$.
	Since $\gamma_t$ is stable, it only contains $\delta$ transversal double points.
	Each double point along $\gamma_t$ has two preimages, so $|D(\gamma_t)|=2\delta$ and the generating function $d_t$ of $D(\gamma_t)$ has degree $2\delta$.
	Therefore, we can write
		\[d_t(s)=a_{2\delta}(t)s^{2\delta}+a_{2\delta+1}(t)s^{2\delta+1}+\dots=a_{2\delta}(t)s^{2\delta}(1+R(s,t))\]
	where $a_{2\delta}(t)\neq 0$. In particular, for $t=0$, we have
		\[d_0(s)=a_{2\delta}(0)s^{2\delta}(1+R(s,0))\]
    The statement then follows from the fact that $\gamma(\mb{C},0)$ is irreducible, hence $\mu=2\delta$ by Milnor's formula (see for instance \cite{Teissier_Invariants}).
\end{proof}

Consider the projections\\
\begin{minipage}{0.45\textwidth}
	\begin{function}
		\pi_1\colon D^2(f) \arrow[r] & \mb{C}^2\\
		(x,y,y') \arrow[r,maps to] & (y,y')
	\end{function}
\end{minipage}%
\hfill%
\begin{minipage}{0.45\textwidth}
	\begin{function}
		\pi_1'\colon D(f) \arrow[r] & \mb{C}\\
		(x,y) \arrow[r,maps to] & y
	\end{function}
\end{minipage}
\noindent\\
These mappings verify that $\pi_1(D^2(f))=D^2(\gamma_t)$ and $\pi_1'(D(f))=D(\gamma_t)$, so we can write $D^2(f)=D^2(\gamma_t)\times_{D(\gamma_t)} D(f)$.
We can now apply \cite{MondNuno_Singularities} Proposition 11.6 to deduce that
	\[D(f)=M_1(\pi)=(\pi'_1)^{-1}(M_1(\pi'))=(\pi'_1)^{-1}(D(\gamma_t))\]
as complex space-germs.
Since $\pi'_1$ is a projection map, $D(\gamma_t)=\pi'_1(D(f))$ as complex space-germs.
If $D(f)$ is generated by $\lambda$ and $D(\gamma_t)$ is generated by $\kappa_t$,
\begin{equation}\label{eq: fibre product equation}
	\lambda=(\pi'_1)^*\kappa_t
\end{equation}

\begin{theorem}\label{thm: equation double point space}
	Let $f\colon (\mb{C}^2,0) \to (\mb{C}^3,0)$ be a generically injective, corank $1$ frontal map-germ.
	If $f$ is given in prenormal form, the generating function of $D(f)$ is given by $\lambda=\tau p_y^2$ for some $\tau \in \ms{O}_2$.
\end{theorem}

\begin{definition}
	We define the \textbf{cuspidal edge} and \textbf{transverse double point} sets of $f$ as
	\begin{align*}
		C(f)=V(p_y); && D_+(f)=V(\tau)
	\end{align*}
\end{definition}

To prove this result, we need to give a lower bound for the Milnor number of a complex, irreducible plane curve $(Y,0)$.
By \cite{deJongPfister}, we can find an $\alpha > 0$ and a $h \in \ms{O}_1$ with $\ord_0(h) > \alpha$ such that $(Y,0)$ is the image of the curve-germ $\gamma\colon (\mb{C},0) \to (\mb{C}^2,0)$ given by $\gamma(t)=(t^\alpha,h(t))$.

\begin{remark}\label{remark: cota milnor}
	With the notation above, let $D_1=\alpha$ and $D_{j+1}$ be the greatest common divisor of $D_j$ and $\beta_j$.
	We have the following identity for $\mu(Y,0)$ due to Milnor (\cite{Milnor_MilnorNumber} Remark 10.10):
		\[\mu(Y,0)=\sum_{j \geq 1}(\beta_j-1)(D_j-D_{j+1})\]
	It is easy to see that, for $\alpha > 1$, $\mu(Y,0) \geq 2(\alpha-1)$.
\end{remark}

\begin{proof}[Proof of Theorem \ref{thm: equation double point space}]
	Let $f\colon N \to Z$ be a representative of $f$ and $f'$ the germ of $f$ at $(x_0,y_0)\in N\backslash \{(0,0)\}$.
	By Theorem \ref{thm: injective prenormal}, $f'$ is $\ms{A}$-equivalent to a germ in the form
		\[g(x,y)=(x,y^m,y^mh(x,y))\]
	where $m \geq 2$ and $h \in \mc{O}_2$.
	The slices of $g$ are then given by the curves $\gamma_t(y)=(y^m,y^mh(t,y))$.
	If $\sigma$ and $\kappa_t$ are the generating functions for $D(g)$ and $D(\gamma_t)$, we can apply Equation \eqref{eq: fibre product equation} to obtain $\sigma=(\pi'_1)^*\kappa_t$.
	By Lemma \ref{lemma: curva milnor divisor}, $\kappa_t(y)=y^\mu$, where $\mu$ is the Milnor number of $\gamma_0$.
	Using Remark \ref{remark: cota milnor},
		\[\sigma(t,y)=y^\mu \implies y^{2(m-1)}|\sigma(t,y)\]

	Throughout the proof of Theorem \ref{thm: injective prenormal}, we see that $p(x,y)-p_0(x)=y^mr(x,y)$, where $p_0(x)=p(x,y_0)$ and $r \in \mc{O}_2$ is a unit.
	Taking derivatives,
		\[p_y(x,y)=y^{m-1}(yr_y(x,y)+mr(x,y))=S(x,y)y^{m-1}\]
	where $S\in \ms{O}_2$ is again a unit.
	Since $g$ is $\ms{A}$-equivalent to $f'$, the germ of $p(x,y)$ at $(x_0,y_0)$ is $\ms{A}$-equivalent to $y^m$.
	If $\eta$ is the generating function for $D(f')$,
		\[\eta(x,y)=r(x,y)y^{2(m-1)}=r(x,y)S^2(x,y)p_y^2(x,y)\]
	If $\lambda$ is the generating function of $D(f)$, $\tau=\lambda/p_y^2$ is a holomorphic function with at most an isolated singularity in $N$.
	Nonetheless, Hartogs' Kugelsatz guarantees that $\tau$ can be uniquely extended onto $(\mb{C}^2,0)$.
	
	We conclude that $\lambda=\tau p_y^2$.
\end{proof}

\begin{table}[ht]
\centering
\begin{tabular}{l|C|C|C}
	\hline
	Singularity		& \text{Parametrization}	& C 		& D_+		\\
	\hline
	\hline
	Cuspidal edge\M	& (x,y^2,y^3)				& y 		& \text{---}\\
	\hline
	\multirow{2}{*}{\makecell[l]{Transverse\\double point}}
				\T	& (x,y,0)					& \text{---}& y 		\\
				\B	& (0,x',y')					& \text{---}& x' 		\\
	\hline
	\makecell[l]{Folded Whitney \T\\umbrella}
				\B	& (x,y^2,xy^3)				& y			& x			\\
	\hline
	Swallowtail	\M	& (x,y^3+3xy,y^4+2xy^2)		& x+y^2		& 3x+y^2	\\
	\hline
	\multirow{2}{*}{\makecell[l]{Cuspidal\\ double point}}
				\T	& (0,x,y)					& \text{---}& x^3-y^2 	\\
				\B	& (x',(y')^2,(y')^3)		& y'		& x'		\\
	\hline
	\multirow{3}{*}{\makecell[l]{Transverse\\triple point}}
				\T	& (x,y,0)					& \text{---}& xy		\\
					& (x',0,y')					& \text{---}& x'y'		\\
				\B	& (0,x'',y'')				& \text{---}& x''y''	
\end{tabular}
\caption{Double point space of the stable frontal surface singularities.
\label{table: stable double point space}}
\end{table}

\begin{example}[Frontalised fold surface]\label{ex: fold double point}
	Let $f\colon (\mb{C}^2,0) \to (\mb{C}^3,0)$ be the fold surface
		\[f(x,y)=(x,y^2,y^3h(x,y^2))\]
	Using Theorem \ref{thm: equation double point space}, we have $D_+(f)=V\left(h(x,y^2)\right)$, which coincides with the double point curve of an $\ms{A}$-finite fold surface.
\end{example}

We now finish this section with a characterisation of $\ms{F}$-finite surfaces in terms of $C(f)$ and $D_+(f)$:
	
\begin{proposition}\label{prop: finite frontal type}
	Let $f\colon (\mb{C}^2,S) \to (\mb{C}^3,0)$ be a corank $1$ frontal surface in prenormal form.
	\begin{enumerate}
		\item If $D(f)$ is generated by $\lambda \in \ms{O}_2$ and $\lambda/p_y$ is regular, $f$ is either a cuspidal edge or a curve of transverse double points. \label{item: cusp transverse}
		\item If $V(p_y,\mu_y)=\{0\}$, $f$ is $\ms{F}$-finite if and only if the critical set of $\lambda/p_y$ is an isolated subset of $(\mb{C}^2,0)$.
	\end{enumerate}
\end{proposition}

A consequence of this statement is that a frontal surface is $\ms{F}$-finite if and only if it only contains cuspidal edges and transversal double points outside the origin.
The six stable frontal surfaces (excluding immersions) are listed on Table \ref{table: stable double point space}, and can be seen on Figure \ref{fig: stable_frontals}.

\begin{proof}[Proof of Proposition \ref{prop: finite frontal type}]
	For the first item, let us assume $\lambda/p_y$ is a regular function.
	If $f$ has a singularity at $0$, $D(f)$ has at least a branch of non-immersive points.
	Regularity of $\lambda/p_y$ implies that this is the only branch of $D(f)$, and therefore $f$ does not have any self-intersections.
	Therefore, $f$ is $1$-to-$1$.

	By Theorem \ref{thm: injective prenormal}, a suitable change of coordinates in the source and target allows us to claim that $p(x,y)=y^m$ and $q(x,y)=y^mh(x,y)$ for some $m \geq 2$ and $h \in \ms{O}_2$.
	In these conditions, Lemma \ref{lemma: curva milnor divisor} tells us that $\lambda(x,y)=y^\mu$, $\mu$ being the Milnor number of the slice $\gamma_0$.
	Regularity of $\lambda/p_y$ then implies that $\delta=1$ and hence $\mu=2$.
	However, since $\mu \geq 2(m-1)$ (Lemma \ref{remark: cota milnor}), we then have that $m=2$.
	Therefore, $\gamma_0$ is a cusp and $f$ can be seen as a $1$-parameter frontal unfolding of $\gamma_0$.

	Since the cusp is stable as a frontal, $f$ is $\ms{A}$-equivalent to the trivial unfolding of the cusp, which is a cuspidal edge.

	If $f$ is not $1$-to-$1$, there is at least one transverse double point in the image, so $D_+(f)\neq\emptyset$.
	Once again, regularity of $\lambda/p_y$ implies that there are no branches of non-immersive points.

	For the second item, let us first assume $f$ is $\ms{F}$-finite.
	By Theorem \ref{thm: frontal mather gaffney}, there exists a representative $f\colon N \to Z$ of $f$ such that $f^{-1}(\{0\})=S$ and the restriction $\tilde f\colon N\backslash S \to Z\backslash\{0\}$ is locally stable as a frontal.
	In particular, we can choose $N$ so that $\tilde f(N\backslash S)$ only contains transverse double points and cuspidal edges at most.

	Since $f$ is a corank $1$ frontal surface, we can take coordinates $(x,y)$ on $N$ and $(X,Y,Z)$ on $Z$ such that $f(x,y)=(x,p(x,y),q(x,y))$ with $p_y|q_y$.
	If $\lambda=0$ is the reduced equation for $D(f)$, $\lambda/p_y\colon N\backslash S \to \mb{C}$ is a regular function, from which it follows that the critical set of $\lambda/p_y$ is contained within $S$.
	Nonetheless, as $S$ is an isolated subset of $N$, so is the critical set of $\lambda/p_y$.

	Conversely, let us assume that the critical set $C$ of $\lambda/p_y$ is an isolated subset of $N$.
	This means that $\lambda/p_y\colon N\backslash C \to \mb{C}$ is a regular function.
	By the previous item, this implies that $f\colon N\backslash C \to Z$ is either a cuspidal edge or a curve of transverse double points, both of which are $\ms{F}$-stable singularities.
	Using Theorem \ref{thm: frontal mather gaffney} once again, we conclude that $f$ is $\ms{F}$-finite.
\end{proof}

\begin{figure}[ht]
\centering
	\includegraphics[width=0.3\textwidth]{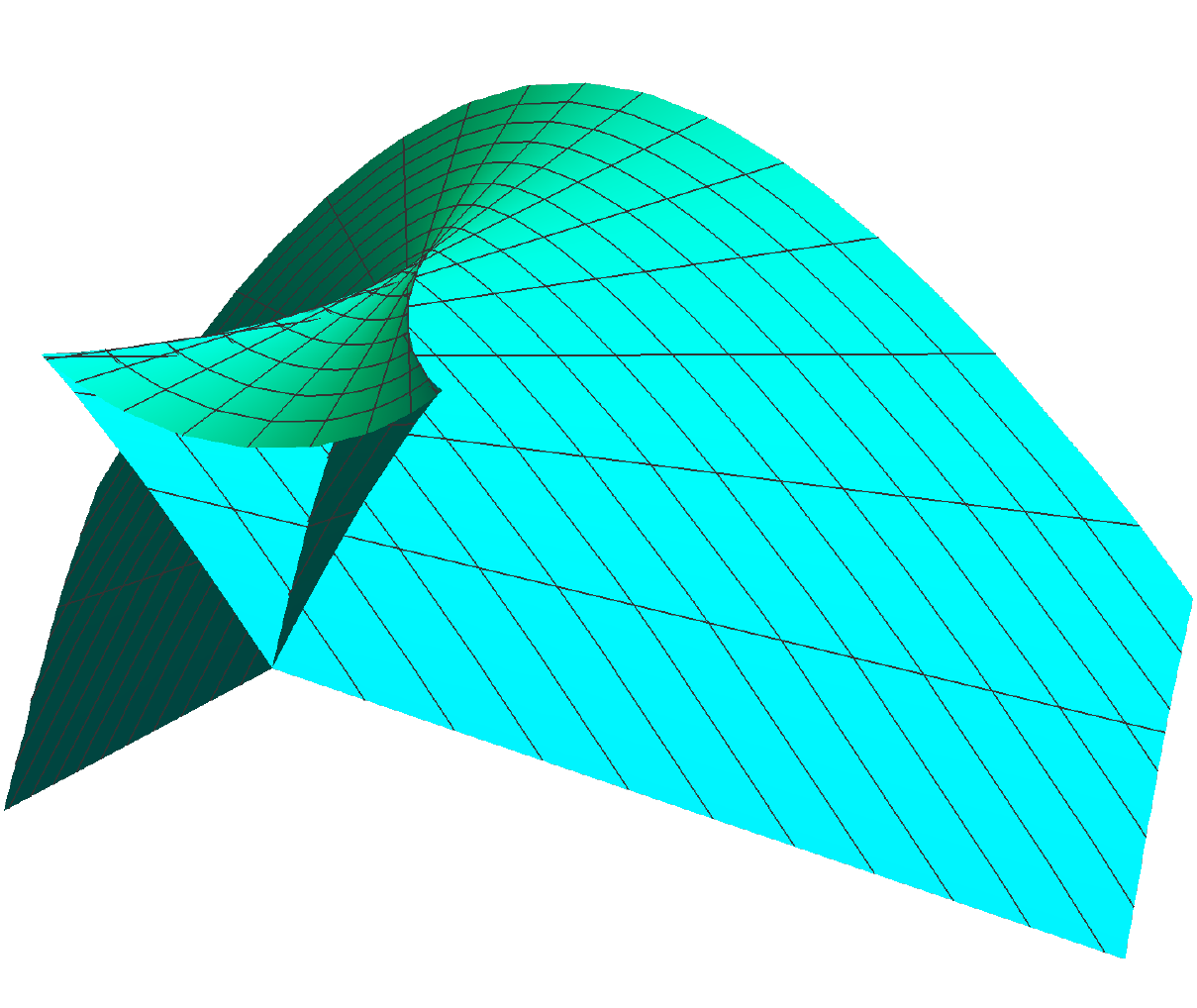}
	\includegraphics[width=0.3\textwidth]{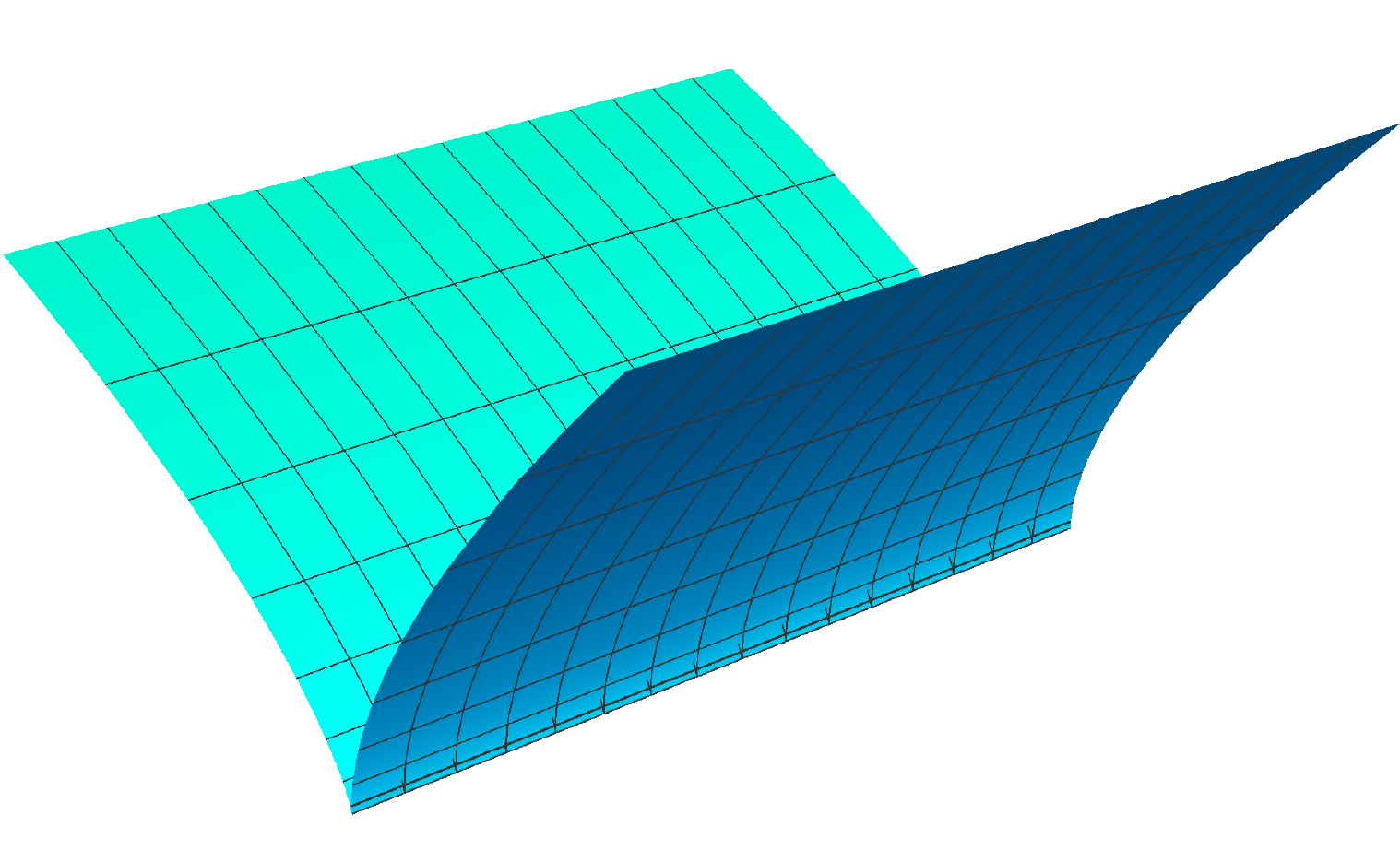}
	\includegraphics[width=0.3\textwidth]{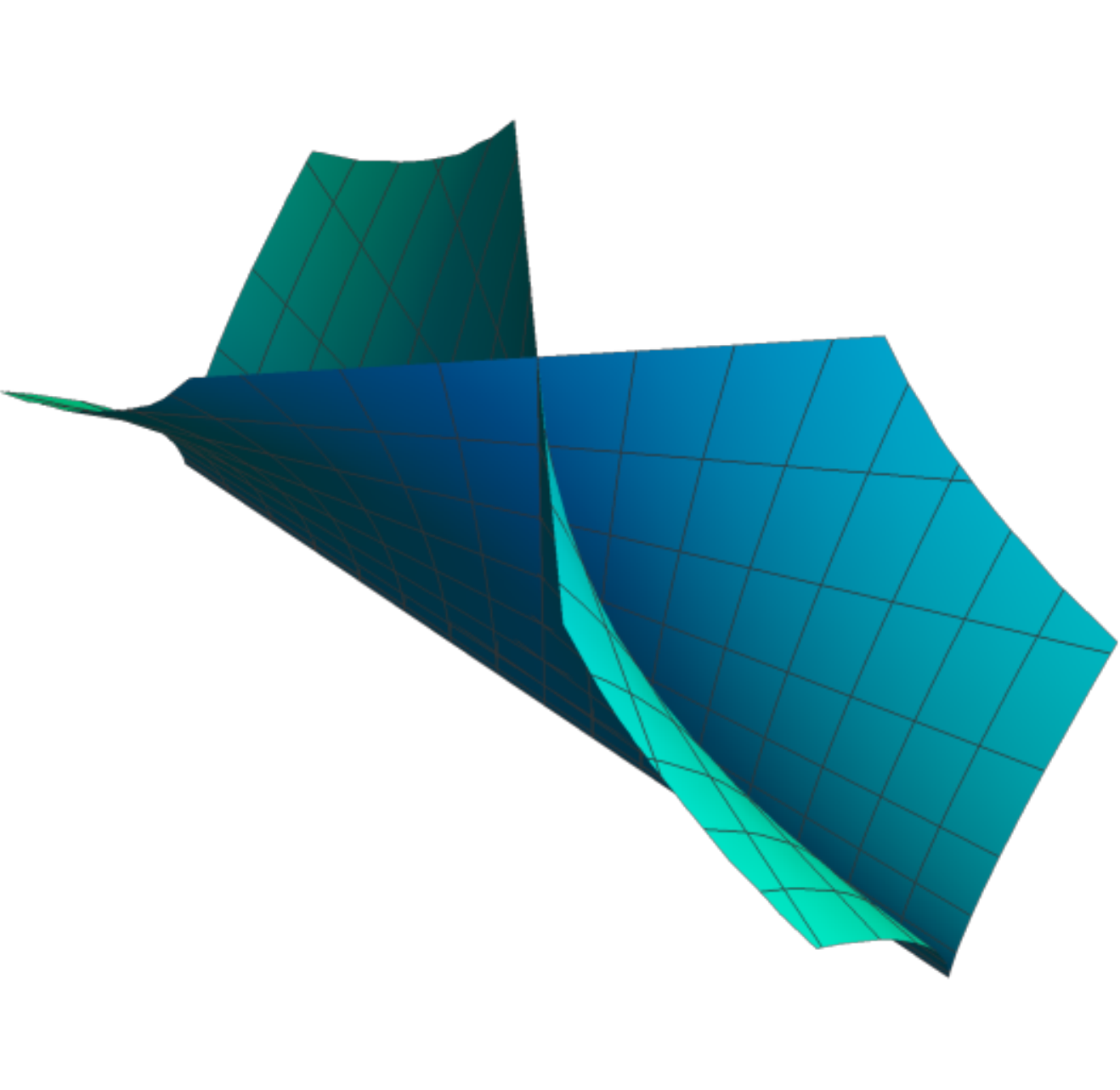}
	
	\includegraphics[width=0.3\textwidth]{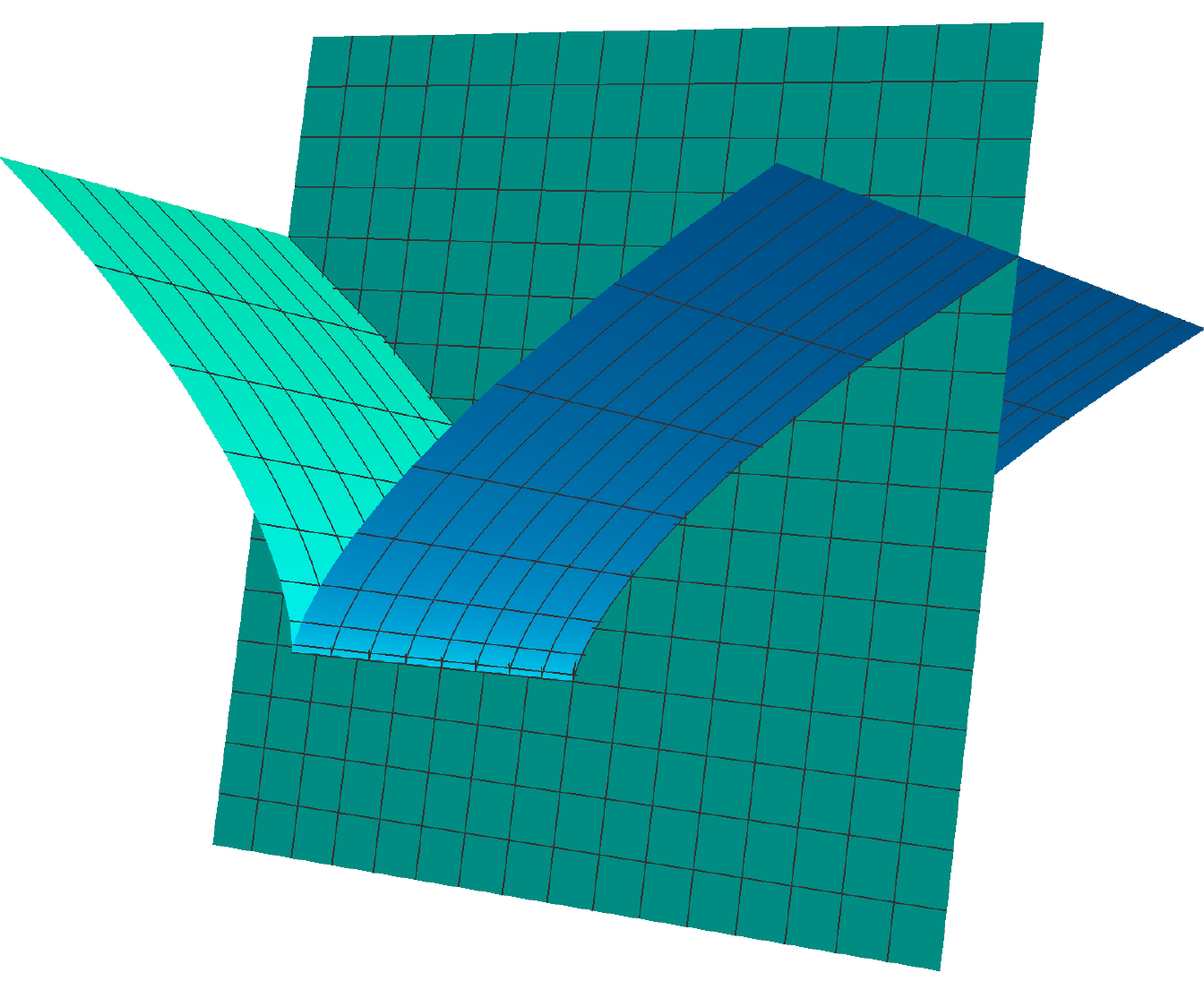}
	\includegraphics[width=0.3\textwidth]{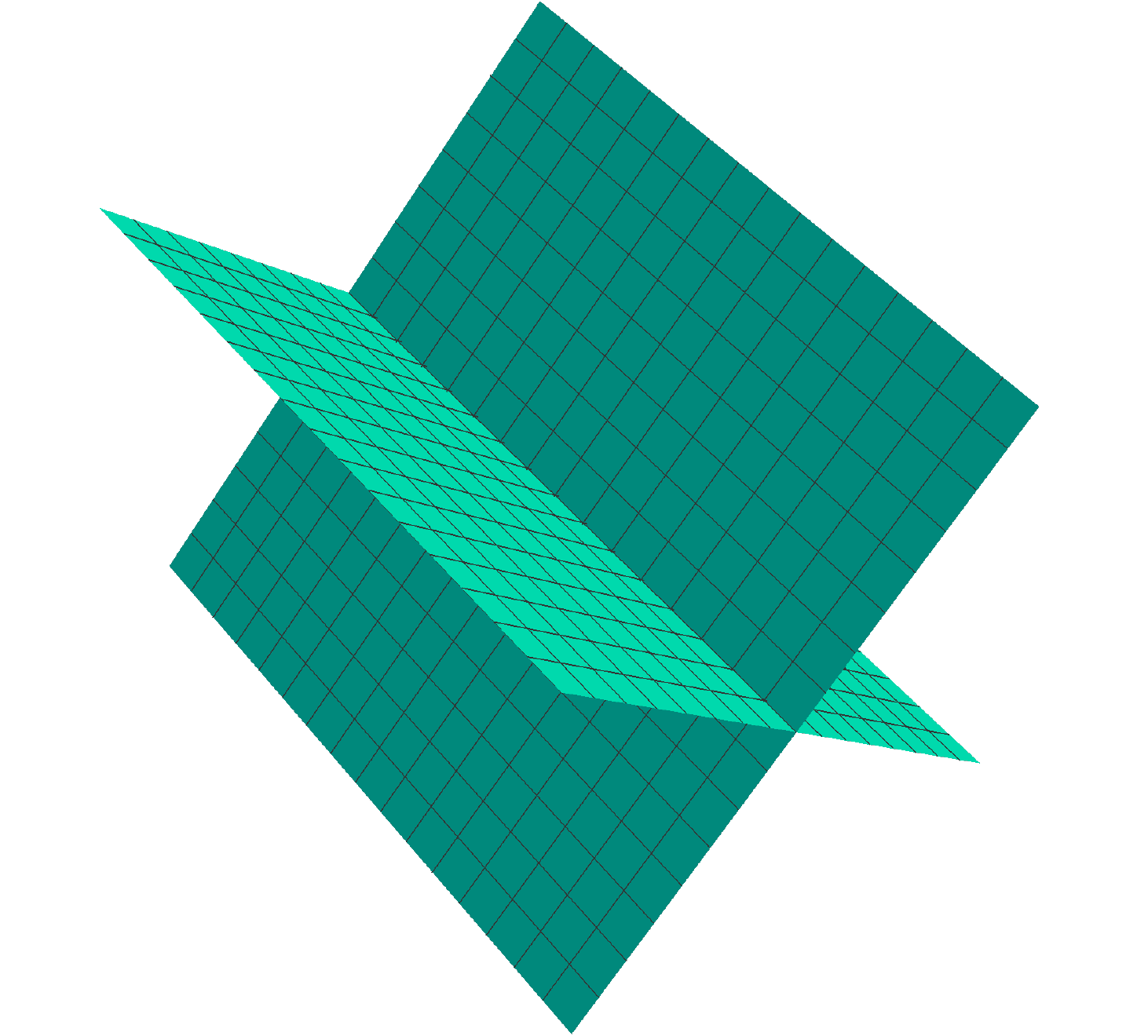}
	\includegraphics[width=0.3\textwidth]{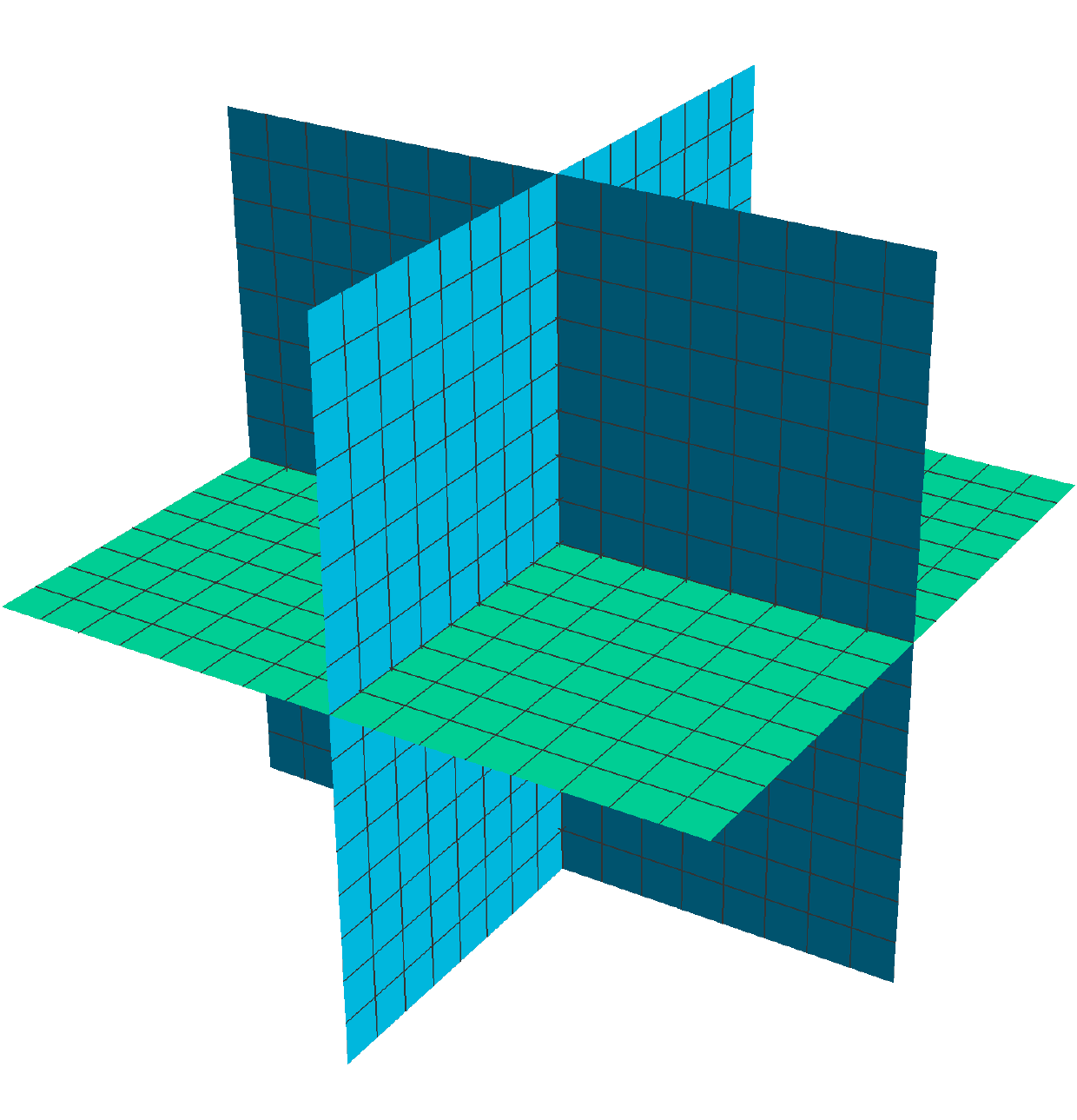}
\caption{Stable frontal surface singularities. \label{fig: stable_frontals}}
\end{figure}

\section{Vanishing homology of a frontal surface}
The formulas from Proposition \ref{prop: S W K T} and Theorems \ref{thm: marar mond} and \ref{thm: milnorf} have been implemented in a \textsc{Singular} library, which can be found at \cite{Singular_frontals}.
	Let $f\colon (\mb{C}^2,S) \to (\mb{C}^3,0)$ be an $\ms{A}$-finite multi-germ, and $f\colon N \to Z$ be a representative of $f$.
A \textbf{stabilisation} of $f$ is a holomorphic family $f_t\colon N_t \to Z$ with $f_0=f$ and $t \in D\subset \mb{C}$ such that the unfolding $F=(f_t,t)$ is stable if $t\neq 0$.
The topological space $\Delta(f)=f_t(N_t)$ is known as the \textbf{disentanglement} of $f$.

Mond \cite{Mond_VanishingCycles} proved that, for a small enough $t \neq 0$, the space $\Delta(f)$ has the homology type of a bouquet of $2$-spheres, the number of which is known as the \textbf{image Milnor number} of $f$ (denoted as $\mu_I(f)$).
As a consequence, the reduced homology of $\Delta(f)$ is known as the \emph{vanishing homology} of $f$.

\subsection{Frontal disentanglement}
\begin{definition}
	Let $f\colon (\mb{C}^n,S) \to (\mb{C}^{n+1},0)$ be a frontal map-germ.
	A smooth family of frontal germs $(f_t)$ is an $\ms{F}$-\textbf{stabilisation} of $f$ if the $1$-parameter unfolding $F=(f_t,t)$ is frontal and $f_t$ is frontal stable for $t \neq 0$.
	Given $t\neq 0$, the \textbf{frontal disentanglement} of $f$ is defined as
		\[\Delta_{\ms{F}}(f)=f_t(\mb{C}^n,S)\]
\end{definition}

Let $f\colon (\mb{C}^2,0) \to (\mb{C}^3,0)$ be a corank $1$ frontal surface with isolated frontal instability (i.e., such that $f(N\backslash\{0\})$ contains at most the singularities described in Table \ref{table: stable double point space}).
By shrinking $N$ if necessary, we can assume that $f(N\backslash\{0\})$ only contains cuspidal edges and transverse double points.

Let $N_0=\{y\in \mb{C}: (x,y) \in N\}$.
We consider the slice $\gamma_0\colon N_0 \to \gamma(N_0)$ of $f$ at $x=0$, given by $\gamma_0(t)=(p(0,t),q(0,t))$.
The assumption that $f$ only has at most cuspidal edges and transverse double points outside the origin implies that the origin is an isolated point in the singular set of $\gamma$.
Using the Mather-Gaffney criterion for $\ms{A}$-equivalence, it follows that $\gamma$ is $\ms{A}$-finite.
However, since $\ms{F}(\gamma)\subseteq \theta(\gamma)$,
	\[\frac{\ms{F}(\gamma)}{T\ms{A}_e\gamma} \subseteq \frac{\theta(\gamma)}{T\ms{A}_e\gamma} \implies \codim_{\ms{F}_e}(\gamma) \leq \codim_{\ms{A}_e}(\gamma) < \infty\]
and thus $\gamma$ is $\ms{F}$-finite.
Therefore, it admits a versal $d$-parameter frontal unfolding, $\Gamma\colon U \to V$.

Since $\Gamma$ is a versal unfolding, it is in particular a stable frontal unfolding of $\gamma$.
Since $f$ is a $1$-parameter unfolding of $\gamma$, $f\times \id$ is a $(d-1)$-parameter unfolding of $\gamma$.
By stability of $\Gamma$, we then have that $f\times \id$ is $\ms{A}$-equivalent to $\Gamma$, from which follows that $\Gamma$ is a stable frontal unfolding of $f$. We then have the following
\begin{proposition}
	Every corank $1$ frontal $f\colon (\mb{C}^2,S) \to (\mb{C}^3,0)$ with isolated $\ms{F}$-instability admits an $\ms{F}$-stabilisation $(f_t)$.
\end{proposition}

Let $f\colon (\mb{C}^2,S) \to (\mb{C}^3,0)$ be a corank $1$ frontal with isolated frontal instability, and let $(f_t)$ be an $\ms{F}$-stabilisation of $f$.
By Theorem \ref{thm: frontal mather gaffney}, if $V(p_y,\mu_y)=\{0\}$ for all $t$, we can find a representative of $\Delta_{\ms{F}}(f)$ that only contains stable frontal singularities (see Table \ref{table: stable double point space}).
We then set the following notation:
\begin{itemize}
	\item $S$: number of swallowtails;
	\item $K$: number of cuspidal double points;
	\item $T$: number of transversal triple points;
	\item $W$: number of folded Whitney umbrellas.
\end{itemize}

Given $h \in \ms{O}_2$, $h_y(x,y)$ is the limit when $y'$ goes to $y$ of the divided difference $h[x,y,y'] \in \ms{O}_3$ (see \S \ref{doublepoints}), so there exist functions $\alpha, \beta \in \ms{O}_3$ such that $ h[x,y,y']\equiv \alpha_y(x,y,y')(y-y') \mod h_y$.
If we now consider $p,q \in \ms{O}_2$ with $p_y|q_y$, we have
    \begin{align*}
	    p[x,y,y'] \equiv \alpha(x,y,y')(y-y') \mod p_y;\\
	    q[x,y,y'] \equiv \alpha'(x,y,y')(y-y') \mod p_y,
	\end{align*}
and thus we consider the ideal $(p_y,\alpha,\alpha') \subseteq \ms{O}_3$.

\begin{proposition}\label{prop: S W K T}
	Let $f\colon (\mb{C}^2,0) \to (\mb{C}^3,0)$ be a frontal mono-germ with isolated instability in the form $f(x,y)=(x,p(x,y),q(x,y))$.
	If the transversal point curve of $f$ is $\tau \in \ms{O}_2$,
	\begin{align*}
		P_3=\dim_{\mb{C}}\frac{\ms{O}_2}{(p_y,p_{yy})}			&=S;		&
		PT=\dim_{\mb{C}}\frac{\ms{O}_2}{(p_y,\tau)}			&=2S+K+W;	\\
		PAA'=\dim_{\mb{C}}\frac{\ms{O}_3}{(p_y,\alpha,\alpha')}	&=2S+K;		&
		F_3=\dim_{\mb{C}}\frac{\ms{O}_3}{\ms{F}_2(f)}			&=T+S+K.
	\end{align*}
\end{proposition}

\begin{proof}
    The ideals $(p_y,p_{yy})$, $(p_y,\tau)$ and $(p_y,\alpha,\alpha')$ are complete intersection in their respective algebras, and thus Cohen-Macaulay.
    Since $f$ has an isolated $\ms{F}$-instability, it follows from Theorem \ref{thm: frontal mather gaffney} and \cite{MondNuno_Singularities} Corollary 11.11 that $\ms{F}_2(f)$ is Cohen-Macaulay.
    By conservation of multiplicity (see \cite{MondNuno_Singularities} Appendix E), we then have that the integers $P_3$, $PT$, $PAA'$ and $F_2$ can be written as a linear combination of $S$, $K$, $T$ and $W$, whose coefficients can be found using the data from Table \ref{table: frontal s k t w}.
\end{proof}

\begin{remark}
    The identity $P_3=S$ can be easily derived for wavefronts using the results from \cite{MararMontaldiRuas_Swallowtail}.
\end{remark}

\begin{table}[ht]
\centering
\begin{tabular}{CC|CCCC}
	\text{Symbol}		& \text{Parametrization}	& S & K & T & W \\
	\hline
	\hline
	\check{S}_k		\M 	& (x,y^2,y^5+x^ky^3)		& 0	& 0 & 0 & k \\
	4^k_1			\M 	& (x,2y^3+x^ky,3y^4+x^ky^2)	& k & 0 & 0 & 0 \\
	5_3				\M 	& (x,5y^4+3xy^2,4y^5+2xy^3)	& 3 & 3 & 0 & 0 \\
	6_1				\M 	& (x,3y^5+xy,5y^6+xy^2)		& 3 & 6 & 1 & 0
\end{tabular}
	\caption{Some corank $1$ frontal surfaces and their dimension 0 singularities in the $\ms{F}$-stabilisation.
	\label{table: frontal s k t w}}
\end{table}

\begin{example}[Frontal fold surfaces]
	Let $f\colon (\mb{C}^2,0) \to (\mb{C}^3,0)$ be a frontal fold surface with prenormal form
	    \[f(x,y)=(x,y^2,y^3h(x,y^2))\]
	and let $\pi\colon (\mb{C}^3,0) \to (\mb{C}^2,0)$ be the projection $\pi(X,Y,Z)=(X,Y)$. The map $\pi\circ f$ has multiplicity $2$, so the presentation matrix for $f$ is $2\times 2$ and $\ms{F}_2(f)=\ms{O}_3$ by definition.
	It follows that $S+K+T=0$ and thus $S=K=T=0$.
    If $i\colon (\mb{C},0) \to (\mb{C}^2,0)$ is the germ of immersion given by $x \mapsto (x,0)$, we conclude using Example \ref{ex: fold double point} that $W=\mult(i^*h)$.
\end{example}

\subsection{Frontal Marar-Mond formulas}
Marar and Mond gave a formula in \cite{MararMond_Formula} that relates the Milnor numbers of the curves $D(f)$ and $f(D(f))$ for an $\ms{A}$-finite $f\colon (\mb{C}^2,0) \to (\mb{C}^3,0)$.
However, it is easy to derive from Theorem \ref{thm: equation double point space} that the Milnor number of $D(f)$ is only finite when $f$ does not feature cuspidal edges on its frontal disentanglement.
Therefore, we need to consider the branches of cuspidal edges and transverse double points separately.

Let $f\colon (\mb{C}^2,0) \to (\mb{C}^3,0)$ be a corank $1$ frontal surface with an isolated frontal instability at $0$ and $(f_t)$ be an $\ms{F}$-stabilisation of $f$.
The double point space $D(f_t)$ then splits into $(D_+)_t$ and $C_t$. If $f_t\colon N_t \to Z_t$ is a representative of $f_t$ with $t \neq 0$, we set the following partition for $N_t$:
\begin{itemize}
	\item $N_t^0$ is the set of $x \in D(f_t)$ where $f_t$ has an isolated singularity;
	\item $N_t^1=D(f_t)\backslash N_t^0$;
	\item $N_t^2$ is the set of $x \in N_t$ such that $f_t$ is immersive at $x$.
\end{itemize}
Since $D(f_t)=C(f_t)\cup D_+(f_t)$, $N_t^1$ splits into the disjoint union of $C(f_t)^1$ and $D_+(f_t)^1$, which we shall simply denote as $C_t^1$ and $D_t^1$.

\begin{lemma}
	Let $f\colon (\mb{C}^2,0) \to (\mb{C}^3,0)$ be a corank $1$ frontal map-germ with isolated instability.
	If $(f_t)$ is an $\ms{F}$-stabilisation of $f$, the projection $\pi_X: X(F) \to (\mb{C},0)$ is flat for $X \in \{C,D\}$.
\end{lemma}

\begin{proof}
	Since $f$ is in prenormal form, we can write
		\[F(x,y,t)=(x,p_t(x,y),q_t(x,y),t)\]
	and we can compute $D(F)$ using the same procedure as described in \S\ref{doublepoints}.
	We then have that $C(F)=V((p_t)_y)$ and $D_+(F)=V(\tau_t)$.
	Both of these spaces are analytic surfaces contained in $\mb{C}^2\times\mb{C}$ with codimension $1$, so they are complete intersections and thus Cohen-Macaulay.

	Let $X=C,D$ and consider the projection $\pi_X\colon X(F) \to \mb{C}$ given by $\pi_X(x,y,t)=t$.
	We have $\pi^{-1}(\{0\})=X(f)$, which is a subspace of codimension $1$ in $X(F)$, which matches the dimension of $(\mb{C},0)$.
	It then follows from \cite{Matsumura} \S 23 that $\pi_X$ is flat.
\end{proof}

\begin{theorem}\label{thm: marar mond}
	Let $f\colon (\mb{C}^2,0) \to (\mb{C}^3,0)$ be a corank $1$ frontal surface with isolated $\ms{F}$-instability.
	If $V(p_y,\mu_y)=\{0\}$,
	\begin{align*}
		\mu(f(C(f)),0)	&=2S+\mu(C(f),0); 	\\
		2\mu(f(D(f)),0)	&=2K+2T+\mu(D(f),0)-W-S+1
	\end{align*}
\end{theorem}

To prove this result, we shall make use of the following
\begin{theorem}[\cite{BuchweitzGreuel_Milnor}]\label{thm: milnor number flat family}
    Let $\pi\colon X \to D\subset \mb{C}$ be a good representative of a flat family $\pi\colon (X,x) \to (D,0)$ of reduced curves. Given $t \in D$,
        \[\mu(X_0,x)-\mu(X_t)=1-\chi(X_t)\]
    where $X_t=\pi^{-1}(t)$ and $\mu(X_t)$ is the sum of the Milnor numbers over the singular points of $X_t$.
\end{theorem}

\begin{proof}[Proof of Theorem \ref{thm: marar mond}]
On the one hand, cuspidal edges are $1$-to-$1$, hence $C_t^1\cong f(C_t^1)$.
Since $\chi(C_t)-\chi(C_t^1)=\chi(f(C_t))-\chi(f(C_t^1))$, it follows that $\chi(C_t)=\chi(f(C_t))$.
On the other hand, transverse double points are 2-to-1, so $\chi(D_t^1)=2\chi(f(D_t^1))$.
Using the identities
\begin{align*}
	\chi(D_t)-\chi(D_t^1)			&=S+W+2K+3T;\\
	\chi(f_t(D_t))-\chi(f(D_t^1))	&=S+W+K+T
\end{align*}
it then follows that $\chi(D_t)-2\chi(f_t(D_t))=T-S-W$.

Using Theorem \ref{thm: milnor number flat family} and the information provided in Table \ref{table: stable double point space}, we have
\begin{align*}
    \mu(C_t)=0&\implies \chi(C_t) = 1-\mu(C,0);\\
    \mu(f_t(C_t))=2S&\implies \chi(f_t(C_t))=2S-\mu(f(C),0)+1;\\
    \mu(D_t)=2K+3T&\implies \chi(D_t)=2K+3T-\mu(D,0)+1;\\
	\mu(f_t(D_t))=2K+2T&\implies 2\chi(f_t(D_t))=4K+4T-2\mu(f(D),0)+2
\end{align*}
The remaining parts are clear.
\end{proof}

\subsection{Frontal Milnor number}
Let $f\colon (\mb{C}^n,S) \to (\mb{C}^{n+1},0)$ be a smooth multi-germ with isolated instability and let $f\colon N \to Z$ be a representative of $f$.
Mond \cite{Mond_VanishingCycles} showed that the disentanglement of $f$ has the homology type of a wedge of $n$-spheres for a small enough $t \neq 0$.
These groups are known as the \textit{vanishing homology} of $f$.

Lê \cite{Le_VanishingCycles} proved that the vanishing homology does not depend on the choice of stabilisation, thus being an invariant for $f$.
The number of spheres is known as the \textbf{image Milnor number} of $f$, $\mu_I(f)$.
Moreover, Lê's result can also be applied in the case when $f$ is frontal and has an isolated frontal instability, allowing us to define the notion of a frontal Milnor number.

\begin{definition}
	Let $f\colon (\mb{C}^n,S) \to (\mb{C}^{n+1},0)$ be a frontal multi-germ with isolated $\ms{F}$-instability.
	We define the \textbf{frontal Milnor number} $\mu_{\ms{F}}(f)$ of $f$ as the number of $n$-spheres in the image of an $\ms{F}$-stabilisation of $f$.
\end{definition}

\begin{proposition}
	Given a non-constant, holomorphic plane curve $\gamma\colon (\mb{C},S) \to (\mb{C}^2,0)$,
		\[\mu_I(\gamma)=\mu_{\ms{F}}(\gamma)+\kappa\]
	where $\kappa$ is the number of cusps in a frontal stabilisation of $\gamma$.
\end{proposition}

\begin{proof}
	Let $\gamma\colon N \to Z$ be a representative of $\gamma$.
	Since $\gamma$ is non-constant, the Curve Selection Lemma implies that $S$ is an isolated subset of $\Sigma(\gamma)$.
	Shrinking $N$ if necessary, we can further assume that $N\cap \Sigma(\gamma)=S$, so $\gamma$ is immersive outside $S$.
	Since immersive maps are both $\ms{A}$-stable and $\ms{F}$-stable, the Mather-Gaffney criterion then gives us Item 1, and Theorem \ref{thm: frontal mather gaffney} gives us Item 2.

	Let $\gamma_t\colon N_t\to Z$ be a frontal stabilisation of $\gamma$.
	The curve $\gamma_t(N_t)$ contains at most normal crossings and plane cusps.
	By the conservation property of the image Milnor number \cite{ConejeroBallesteros_imageMilnor}, we have
		\[\mu_I(\gamma)=\mu_{\ms{F}}(\gamma)+\sum_{\mathclap{x \in \Sigma(\gamma_t)}}\mu_I(\gamma_t;\gamma_t(x))\]

	Let $x \in \Sigma(\gamma_t)$.
	Since $\gamma_t$ is $\ms{F}$-stable, the germ of $\gamma_t$ at $x$ is a cusp, so $\mu_I(\gamma_t;\gamma_t(x))=1$ and the statement follows.
\end{proof}

\begin{minipage}{.34\textwidth}
	\centering
	\includegraphics[width=\textwidth]{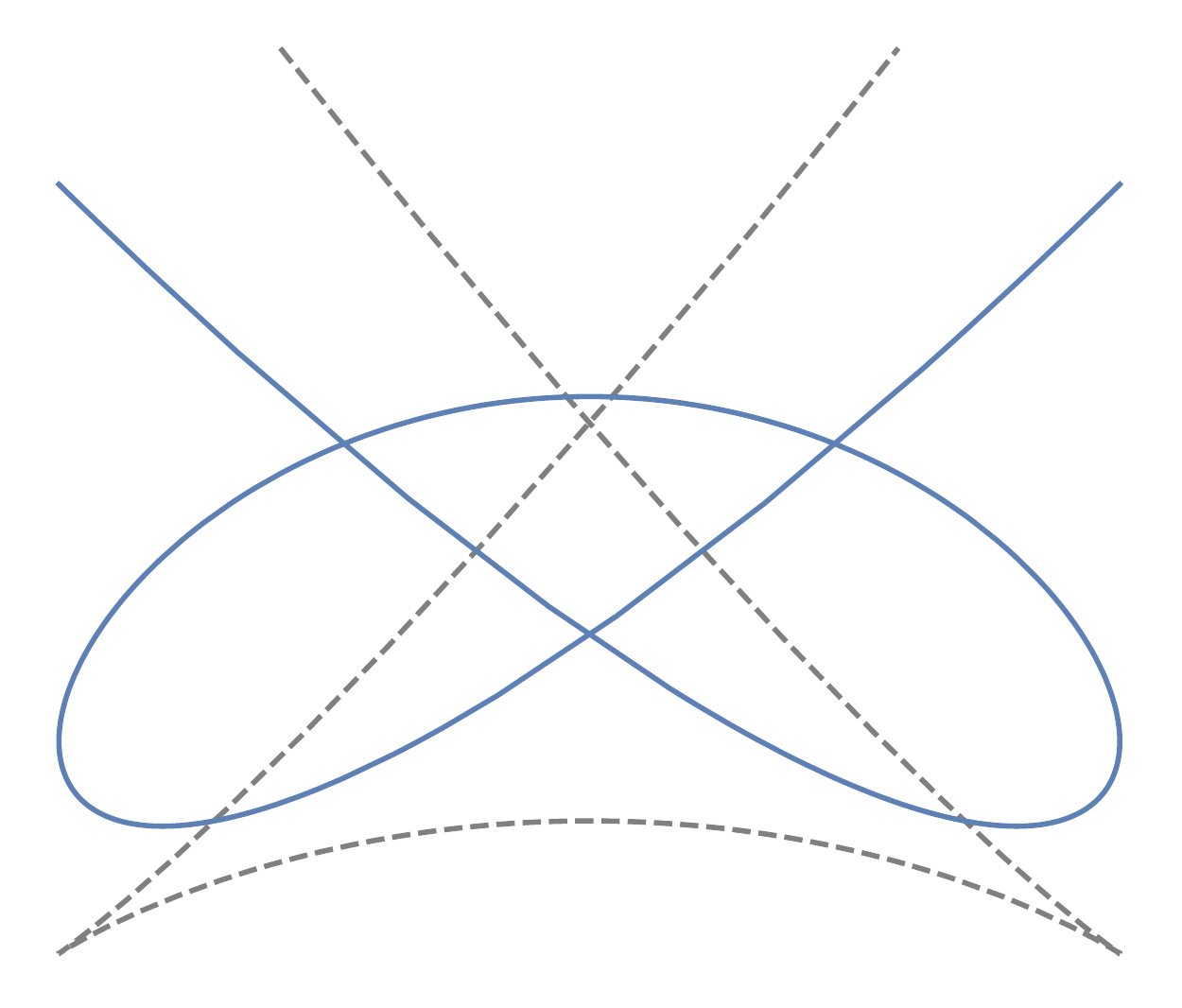}
\end{minipage}%
\hfill%
\begin{minipage}{.55\textwidth}
	\captionof{figure}{Frontal stabilisation of the $E_6$ singularity (dashed, grey) along with a stabilisation (solid, blue).\label{fig: E6 milnor}}
\end{minipage}

We now move onto frontal surfaces.
Let $f\colon (\mb{C}^2,0) \to (\mb{C}^3,0)$ be a corank $1$ frontal map with isolated instability, $(f_t)$ be an $\ms{F}$-stabilisation for $f$ and $F=(f_t,t)$.
We consider the function $\pi\colon F(\mb{C}^2\times\mb{C},0) \to (\mb{C},0)$ given by $\pi(X,Y,Z,t)=t$.
A result by L{\^e} \cite{Le_VanishingCycles} states that the fibres of $\pi$ have the homotopy type of a bouquet of $2$-spheres; however, one of such fibres is $\Delta_\ms{F}(f)$.
Therefore,
    \[\mu_{\ms{F}}(f)=\chi(f_t(N_t))-1=\chi(f_t(N_t^2))+\chi(f_t(N_t^1)+\chi(f_t(N_t^0)-1\]

\begin{theorem}\label{thm: milnorf}
	Given a corank $1$ frontal $f\colon (\mb{C}^2,0) \to (\mb{C}^3,0)$ with isolated $\ms{F}$-instability,
	\begin{align*}
		\mu_{\ms{F}}(f)	&=\mu(f(D(f)),0)-S-W+T+1=\\
						&=\frac{1}{2}\Big(\mu(D(f),0)+3(1-S-W)\Big)+K+2T
	\end{align*}
\end{theorem}

\begin{proof}
	We can assume without any loss of generality that $N_t$ is an open, convex neighbourhood of $0$, in which case $\chi(N_t^2)=1-\chi(N_t^0)-\chi(N_t^1)$.
	Since $f_t(N_t^2) \cong N_t^2$ by construction,
		\[\mu_{\ms{F}}(f)=\chi(f_t(N_t^1))-\chi(N_t^1)+\chi(f_t(N_t^0))-\chi(N_t^0)\]

	Recall that $N_t^1=C(f_t)^1\sqcup D(f_t)^1$.
	Since $C_t^1\cong f_t(C_t^1)$ and $D_t^1$ is a double cover for $f_t(D_t^1)$,
	\begin{align*}
		\chi(f_t(N_t)^0)-\chi(N_t^0)&=-K-2T;\\
		\chi(f_t(N_t)^1)-\chi(N_t^1)&=\chi(f_t(D_t))-S-W-K-T;
	\end{align*}
	from which follows that $\mu_{\ms{F}}(f)=\chi(f_t(D_t))-S-W-2K-3T$.
	Using Theorem \ref{thm: milnor number flat family}, we have $\chi(f_t(D_t))=2K+2T-\mu(f(D),0)+1$, from which follows that
		\[\mu_{\ms{F}}(f)=\mu(f(D),0)-S-W+T+1\]
\end{proof}

We finish this section by proposing a frontal version of Mond's conjecture \cite{Mond_Conjecture}:
\begin{conjecture}\label{mond frontal}
	Let $f\colon (\mb{C}^n,S) \to (\mb{C}^{n+1},0)$ be an $\ms{F}$-finite frontal map.
	Then $\mu_{\ms{F}}(f) \geq \codim_{\ms{F}_E}(f)$, with equality if and only if $f$ is quasihomogeneous.
\end{conjecture}

	\bibliographystyle{plain}
	\bibliography{references}
\end{document}